\documentclass[12pt]{article}

\usepackage{amsfonts,amssymb,amsmath,amsthm,epsfig,euscript}

\newtheorem{theorem}{Theorem}
\newtheorem{lemma}[theorem]{Lemma}
\newtheorem{corollary}[theorem]{Corollary}
\newtheorem{definition}[theorem]{Definition}

\long\def\symbolfootnote[#1]#2{\begingroup
\def\thefootnote{\fnsymbol{footnote}}\footnote[#1]{#2}\endgroup}

\newcommand{\la}{\lambda}
\newcommand{\La}{\Lambda}
\newcommand{\ris}[1]{\mathrm{ris}(#1)}
\newcommand{\wris}[1]{\mathrm{wris}(#1)}
\newcommand{\sris}[1]{\mathrm{sris}(#1)}
\newcommand{\Ris}{\mathrm{Ris}}
\newcommand{\WRis}{\mathrm{WRis}}
\newcommand{\SRis}{\mathrm{SRis}}

\newcommand{\wdes}[1]{\mathrm{wdes}(#1)}
\newcommand{\sdes}[1]{\mathrm{sdes}(#1)}
\newcommand{\des}[1]{\mathrm{des}(#1)}

\newcommand{\qbinom}[2]{\genfrac{[}{]}{0pt}{}{#1}{#2}_{q}}
\newcommand{\pqbinom}[2]{\genfrac{[}{]}{0pt}{}{#1}{#2}_{p,q}}
\newcommand{\rbinom}[2]{\genfrac{[}{]}{0pt}{}{#1}{#2}_{r}}
\newcommand{\sg}{\sigma}
\newcommand{\ep}{\epsilon}

\newcommand{\cref}[1]{Corollary \ref{corollary:#1}}

\newcommand{\inv}[1]{\mathrm{inv}(#1)}
\newcommand{\coinv}[1]{\mathrm{coinv}(#1)}
\newcommand{\qbin}[3]{\genfrac{[}{]}{0pt}{}{#1}{#2}_{#3}}
\newcommand{\red}[1]{\mathrm{red}(#1)}

\newcommand{\tmch}[1]{\text{$\tau$-$\mathrm{mch}$}(#1)}
\newcommand{\umch}[1]{\text{$u$-$\mathrm{mch}$}(#1)}

\newcommand{\tumch}[1]{\text{$(\tau,u)$-$\mathrm{mch}$}(#1)}
\newcommand{\tulap}[1]{\text{$(\tau,u)$-$\mathrm{nlap}$}(#1)}
\newcommand{\Umch}[1]{\text{$\Upsilon$-$\mathrm{mch}$}(#1)}
\newcommand{\Ulap}[1]{\text{$\Upsilon$-$\mathrm{nlap}$}(#1)}

\newcommand{\fig}[2]{\begin{figure}[ht]
\centerline{\scalebox{.66}{\epsfig{file=#1.eps}}}
\caption{#2}
\label{figure:#1}
\end{figure}}

\title{New pattern matching conditions for wreath products
of the cyclic groups with symmetric groups.}

\author{
Sergey Kitaev\footnote{The work presented here was supported by grant no. 090038011 from the Icelandic Research Fund.}\\
\small The Mathematics Institute\\[-0.8ex]
\small School of Computer Science\\[-0.8ex]
\small Reykjav\'{i}k University\\[-0.8ex]
\small IS-103 Reykjav\'{i}k, Iceland\\[-0.8ex]
\small \texttt{sergey@ru.is}
\and
Andrew Niedermaier \\
\small Department of Mathematics\\[-0.8ex]
\small University of California, San Diego\\[-0.8ex]
\small La Jolla, CA 92093-0112. USA\\[-0.8ex]
\small \texttt{aniederm@math.ucsd.edu}
\and
Jeffrey Remmel\footnote{Partially supported by NSF grant DMS 0654060.} \\
\small Department of Mathematics\\[-0.8ex]
\small University of California, San Diego\\[-0.8ex]
\small La Jolla, CA 92093-0112. USA\\[-0.8ex]
\small \texttt{remmel@math.ucsd.edu}
\and
Manda Riehl\footnote{Partially supported by a grant from the Office of Research and Sponsored Programs, UWEC.}\\
\small Department of Mathematics\\[-0.8ex]
\small University of Wisconsin, Eau Claire\\[-0.8ex]
\small Eau Claire, WI 54702-4004 USA\\[-0.8ex]
\small \texttt{riehlar@uwec.edu}
\and
}

\date{\small Submitted: Date 1;  Accepted: Date 2;
 Published: Date 3.\\
\small MR Subject Classifications: 05A15, 05E05}

\begin{document}
\maketitle

\begin{abstract}
\noindent We present several multi-variable generating functions for
a new pattern matching condition on the wreath product $C_k \wr S_n$
of the cyclic group $C_k$ and the symmetric group $S_n$. Our new
pattern matching condition requires that the underlying permutations
match in the usual sense of pattern matching for $S_n$ and that the
corresponding sequence of signs match in the sense of words, rather
than the exact equality of signs which has been previously studied.
We produce the generating functions for the number of matches that
occur in elements of $C_k \wr S_n$ for any pattern of length $2$ by
applying appropriate homomorphisms from the ring of symmetric
functions over an infinite number of variables to simple symmetric
function identities. We also provide multi-variable generating
functions for the distribution of nonoverlapping matches and for the
number of elements of $C_k \wr S_n$ which have exactly $2$ matches
which do not overlap for several patterns of length $2$.
\end{abstract}

\section{Introduction}

The goal of this paper is to study pattern matching conditions on
the wreath product $C_k \wr S_n$ of the cyclic group $C_k$ and the
symmetric group $S_n$. $C_k \wr S_n$ is the group of $k^nn!$ signed
permutations where we allow $k$ signs of the form $1=\omega^0$, $\omega$,
$\omega^2$, $\ldots$, $\omega^{k-1}$ for some  primitive
$k$-th root of unity $\omega$. We can think of the elements $C_k \wr S_n$ as
pairs $\gamma = (\sg,\epsilon)$ where $\sg  = \sg_1 \ldots \sg_n \in
S_n$ and $\ep =\ep_1 \ldots \ep_n \in \{1,\omega,\ldots,
\omega^{k-1}\}^n$. For ease of notation, if $\epsilon =
(\omega^{w_1},\omega^{w_2}, \ldots, \omega^{w_n})$ where $w_i \in
\{0, \ldots, k-1\}$ for $i =1, \ldots, n$, then we simply write
$\gamma = (\sg,w)$ where $w = w_1 w_2 \ldots w_n$.

Given a sequence $\sg = \sg_1 \cdots \sg_n$ of distinct integers,
let $\red{\sg}$ be the permutation found by replacing the
$i^{\textrm{th}}$ largest integer that appears in $\sg$ by $i$.  For
example, if $\sg = 2~7~5~4$, then $\red{\sg} = 1~4~3~2$.  Given a
permutation $\tau$ in the symmetric group $S_j$, define a
permutation $\sg = \sg_1 \cdots \sg_n \in S_n$ to have a {\em
$\tau$-match at place $i$} provided $\red{\sg_i \cdots \sg_{i+j-1}}
= \tau$.  Let $\tmch{\sg}$ be the number of $\tau$-matches in the
permutation $\sg$.  Similarly, we say that $\tau$ {\em occurs} in
$\sg$ if there exist $1 \leq i_1 < \cdots < i_j \leq n$ such that
$\red{\sg_{i_1} \cdots \sg_{i_j}} = \tau$.  We say that $\sg$ {\em
avoids} $\tau$ if there are no occurrences of $\tau$ in $\sg$.

We can define similar notions for words over a finite alphabet $[k]=
\{0,1, \ldots, k-1\}$. Given a word  $w = w_1 \cdots w_n \in [k]^n$,
let $\red{w}$ be the word found by replacing the $i^{\textrm{th}}$
largest integer that appears in $w$ by $i-1$.  For example, if $w =
2~7~2~4~7$, then $\red{w} = 0~2~0~1~2$.  Given a word $u \in [k]^j$
such that $\red{u}=u$, define a word $w \in [k]^n$ to have a {\em
$u$-match at place $i$} provided $\red{w_i \cdots w_{i+j-1}} = u$.
Let $\umch{w}$ be the number of $u$-matches in the word  $w$.
Similarly, we say that $u$ {\em occurs} in a word $w$ if there exist
$1 \leq i_1 < \cdots < i_j \leq n$ such that $\red{w_{i_1} \cdots
w_{i_j}} = u$. We say that $w$ {\em avoids} $u$ if there are no
occurrences of $u$ in $w$.

There are a number of papers on pattern matching and pattern
avoidance in $C_k\wr S_n$ ~\cite{E,Man,Man1,MW}.  For example, the
following pattern matching condition was studied in \cite{Man, Man1,MW}.
\begin{definition}\label{def1}
\begin{enumerate}
\item We say that an element $(\tau,u) \in C_k \wr S_j$ {\bf occurs} in
an element $(\sg,w) \in C_k \wr S_n$ if there are $1 \leq i_1 < i_2
< \cdots < i_j \leq n$ such that $\red{\sg_{i_1} \ldots \sg_{i_j}} =
\tau$  and $w_{i_p}=u_p$ for $p= 1, \ldots, j$.

\item We say that an element
$(\sg,w) \in C_k \wr S_n$ {\bf avoids} $(\tau,u) \in C_k \wr S_j$ if
there are no occurrences of $(\tau,u)$ in $(\sg,w)$.

\item If $(\sg,w) \in C_k \wr S_n$  and $(\tau,u) \in C_k \wr S_j$, then
we say that there is a {\bf $(\tau,u)$-match in $(\sg,w)$ starting
at position $i$} if $\red{\sg_{i} \sg_{i+1} \ldots \sg_{i+j-1}} =
\tau$ and $w_{i+p-1} = u_{p}$ for $p =1, \ldots, j$.
\end{enumerate}
\end{definition}

\noindent That is, an occurrence or match of $(\tau,u) \in C_k \wr S_j$ in an element
$(\sg,w) \in C_k \wr S_n$ is just an ordinary occurrence or match of $\tau$ in
$\sg$ where the corresponding signs agree exactly.
For example, Mansour \cite{Man1} proved via recursion that
for any $(\tau, u) \in C_k \wr S_2$, the number of
$(\tau,u)$-avoiding elements in $C_k \wr S_n$ is
$\sum_{j=0}^n j!(k-1)^j \binom{n}{j}^2$.  This generalized
a result of Simion \cite{Sim} who proved the same result
for the hyperoctrahedral group $C_2 \wr S_n$.
Similarly,
Mansour and West \cite{MW} determined
the number of permutations in $C_2 \wr S_n$ that avoid all possible
2 and 3 element set of patterns of elements of
$C_2 \wr S_2$.  For example, let
$K_n^1$, the
number of $(\sg,\epsilon) \in C_2 \wr S_n$ that avoid all the
patterns in the set $\{(1~2,0~0),(1~2,0~1),(2~1,1~0)\}$, $K_n^2$, the
number of $(\sg,\epsilon) \in C_2 \wr S_n$ that avoid all the
patterns in the set $\{(1~2,0~1),(1~2,1~0),(2~1,0~1)\}$, and $K_n^3$, the
number of $(\sg,\epsilon) \in C_2 \wr S_n$ that avoid all the
patterns in the set $\{(1~2,0~0),(1~2,0~1),(2~1,0~0)\}$. They proved
that
\begin{eqnarray*}
K_n^1 &=& F_{2n+1}, \\
K_n^2 &=&   n! \sum_{j=0}^n \binom{n}{j}^{-1}, \ \mbox{and} \\
K_n^3 &=& n! +n!\sum_{j=1}^n \frac{1}{j}
\end{eqnarray*}
where $F_n$ is $n$-th Fibonacci number.

In this paper, we shall drop the requirement of the exact matching
of signs and replace it by the condition that the two sequences of
signs match in the sense of words described above. That is, we shall
consider the following pattern matching conditions:
\begin{definition}\label{def2}
Suppose that $(\tau,u) \in C_k \wr S_j$ and
$\red{u} =u$.
\begin{enumerate}
\item We say that $(\tau,u)$ {\bf bi-occurs} in
$(\sg,w) \in C_k \wr S_n$ if there are
$1 \leq i_1 < i_2 < \cdots < i_j \leq n$ such that
$\red{\sg_{i_1} \ldots \sg_{i_j}} = \tau$  and
$\red{w_{i_1} \ldots w_{i_j}} = u$.

\item We say that an element
$(\sg,w) \in C_k \wr S_n$ {\bf bi-avoids} $(\tau,u)$ if there are no
bi-occurrences of $(\tau,u)$ in $(\sg,w)$.

\item  We say that there is a {\bf $(\tau,u)$-bi-match in $(\sg,w) \in
C_k \wr S_n$ starting at position $i$} if $\red{\sg_{i} \sg_{i+1}
\ldots \sg_{i+j-1}} = \tau$ and $\red{w_{i} w_{i+1} \ldots
w_{i+j-1}} = u$.
\end{enumerate}
\end{definition}
\noindent For example, suppose that $(\tau,u) = (1~2,0~0)$ and
$(\sg,w) = (1~3~2~4,1~2~2~2)$. Then there are no occurrences or
matches of $(\tau,u)$ in $(\sg,w)$ according to Definition
\ref{def1}. However, there is a $(\tau,u)$-bi-match in $(\sg,w)$
starting at position $3$; additionally, 34 and 24 in $\sg$ are
bi-occurrences of $(\tau,u)$ in $(\sg,w)$. Let $\tumch{(\sg,w)}$
be the number of $(\tau,u)$-bi-matches in $(\sg,w) \in C_k \wr
S_n$. Let $\tulap{(\sg,w)}$ be the maximum number of
non-overlapping $(\tau,u)$-bi-matches in $(\sg,w)$ where two
$(\tau,u)$-bi-matches are said to overlap if there are positions
in $(\sg,w)$ that are involved in both $(\tau,u)$-bi-matches.

One can easily extend these notions to sets of elements of $C_k
\wr S_j$.  That is, suppose that $\Upsilon \subseteq C_k \wr S_j$
is such that every $(\tau,u) \in \Upsilon$ has the property that
$\red{u} =u$. Then $(\sg,w)$ has a $\Upsilon$-bi-match at place
$i$ provided $(\red{\sg_i \ldots \sg_{i+j-1}},\red{w_i \ldots
w_{i+j-1}}) \in \Upsilon$. Let $\Umch{(\sg,w)}$ and
$\Ulap{(\sg,w)}$ be the number of $\Upsilon$-bi-matches and
non-overlapping $\Upsilon$-bi-matches in $(\sg,w)$, respectively.

In this paper, we shall mainly study the distribution of bi-matches
for patterns of length 2, i.e. where $(\tau,u) \in C_k \wr S_2$.
This is closely related to the analogue of rises and descents in
$C_k \wr S_n$ where we compare pairs using the product order. That
is, instead of thinking of an element of $C_k \wr S_n$ as a pair
$(\sg_1 \ldots \sg_n,w_1 \ldots w_n)$, we can think of it as a
sequence of pairs $(\sg_1,w_1) (\sg_2,w_2) \ldots (\sg_n,w_n)$. We
then define a partial order on such pairs by the usual product
order. That is, $(i_1,j_1) \leq (i_2,j_2)$ if and only if $i_1 \leq
i_2$ and $j_1 \leq j_2$.  Then we define the following statistics
for elements $(\sg,w) \in C_k \wr S_n$.
\begin{eqnarray*}
Des((\sg,w)) &=& \{i:\sg_i > \sg_{i+1} \ \& \ w_i \geq w_{i+1}\}, \des{(\sg,w)} =
|Des((\sg,w))|, \\
Ris((\sg,w)) &=& \{i:\sg_i < \sg_{i+1} \ \& \ w_i \leq w_{i+1}\}, \ris{(\sg,w)} =
|Ris((\sg,w))|, \\
WDes((\sg,w)) &=& \{i:\sg_i > \sg_{i+1} \ \& \ w_i = w_{i+1}\}, \wdes{(\sg,w)} =
|WDes((\sg,w))|, \\
WRis((\sg,w)) &=& \{i:\sg_i < \sg_{i+1} \ \& \ w_i = w_{i+1}\}, \wris{(\sg,w)} =
|WRis((\sg,w))|, \\
SDes((\sg,w)) &=& \{i:\sg_i > \sg_{i+1} \ \& \ w_i > w_{i+1}\}, \sdes{(\sg,w)} =
|SDes((\sg,w))|, \\
SRis((\sg,w)) &=& \{i:\sg_i < \sg_{i+1} \ \& \ w_i < w_{i+1}\}, \sris{(\sg,w)} =
|SRis((\sg,w))|.
\end{eqnarray*}
We shall refer to $Des((\sg,w))$ as the {\em descent set} of $(\sg,w)$,
$WDes((\sg,w))$ as the {\em weak descent set} of $(\sg,w)$, and
$SDes((\sg,w))$ as the {\em strict descent set} of $(\sg,w)$. Similarly,
we shall refer to $Ris((\sg,w))$ as the {\em rise set} of $(\sg,w)$,
$WRis((\sg,w))$ as the {\em weak rise set} of $(\sg,w)$, and
$SRis((\sg,w))$ as the {\em strict rise set} of $(\sg,w)$.
It is easy to see that
\begin{description}
\item $i \in WDes((\sg,w))$ if and only if there is
a $(2~1,0~0)$-bi-match starting at position $i$,

\item $i \in SDes((\sg,w))$ if and only if there is
a $(2~1,1~0)$-bi-match starting at position $i$, and

\item $i \in Des((\sg,w))$ if and only if there is
a $\Upsilon$-bi-match starting at position $i$ where $\Upsilon =
\{(2~1,0~0), (2~1,1~0)\}$.
\end{description}
Similarly,
\begin{description}
\item $i \in WRis((\sg,w))$ if and only if there is
a $(1~2,0~0)$-bi-match starting at position $i$,

\item $i \in SRis((\sg,w))$ if and only if there is
a $(1~2,0~1)$-bi-match starting at position $i$, and

\item $i \in Ris((\sg,w))$ if and only if there is
a $\Upsilon$-bi-match starting at position $i$ where $\Upsilon =
\{(1~2,0~0), (1~2,0~1)\}$.
\end{description}

If $\sg = \sg_1 \ldots \sg_n \in S_n$, then we define
the reverse of $\sg$, $\sg^r$, by  $\sg^r  = \sg_n \ldots \sg_1$.
Similarly, if $w = w_1 \ldots w_n \in [k]^n$, then we define
$w^r = w_n \ldots w_1$. It is easy to see
that
\begin{eqnarray*}
\ris{(\sg,w)} &=& \des{(\sg^r,w^r)}, \\
\wris{(\sg,w)} &=& \wdes{(\sg^r,w^r)}, \ \mbox{and} \\
\sris{(\sg,w)} &=& \sdes{(\sg^r,w^r)}.
\end{eqnarray*}
Thus we need to find the distributions for only one of the
corresponding pairs. We shall prove the following generating
functions.

\begin{equation}\label{ris}
\sum_{n \geq 0} \frac{t^n}{n!} \sum_{(\sg,w) \in C_k \wr S_n}
x^{\ris{(\sg,w)}} = \frac{1-x}{1-x + \sum_{n\geq 1}\frac{((x-1)t)^n}{n!} \binom{n+k-1}{n}}.
\end{equation}

\begin{equation}\label{wris}
\sum_{n \geq 0} \frac{t^n}{n!} \sum_{(\sg,w) \in C_k \wr S_n}
x^{\wris{(\sg,w)}} = \frac{1-x}{1-x +k(e^{(x-1)t}-1)}.
\end{equation}

\begin{equation}\label{sris}
\sum_{n \geq 0} \frac{t^n}{n!} \sum_{(\sg,w) \in C_k \wr S_n}
x^{\sris{(\sg,w)}} = \frac{1-x}{1-x + \sum_{n\geq 1}\frac{((x-1)t)^n}{n!} \binom{k}{n}}.
\end{equation}

Other distributions results for $(\tau,u)$-bi-matches follow
from these results.
For example, if $\sg = \sg_1 \ldots \sg_n \in S_n$, then
we define the complement of $\sg$, $\sg^c$, by
$$\sg^c = (n+1 -\sg_1) \ldots (n+1-\sg_n).$$ If
$w =w_1 \ldots w_n \in [k]^n$, then
we define the complement of $w$, $w^c$, by
$$w^c = (k-1 -w_1) \ldots (k-1-w_n).$$
We can then
consider maps $\phi_{a,b}:C_k \wr S_n \rightarrow C_k \wr S_n$ where
$\phi_{a,b}((\sg,w)) = (\sg^a,w^b)$ for $a,b \in \{r,c\}$. Such maps
will easily allow us to establish that the distribution of
$(\tau,u)$-bi-matches is the same for various classes of $(\tau,u)$'s. For
example, one can use such maps to show that
the distributions of $(1~2,0~1)$-bi-matches, $(2~1,0~1)$-bi-matches,
$(1~2,1~0)$-bi-matches, and $(2~1,1~0)$-bi-matches are all the same.

Another interesting case is when we let
$\Upsilon = \{(1~2,0~1),(1~2,1~0)\}$. In this case
we have a $\Upsilon$-bi-match in $(\sg,w)$ starting at $i$ if and only
if $\sg_i < \sg_{i+1}$ and $w_i \neq w_{i+1}$.  In that case,
we shall show that

\begin{equation}\label{Up}
\sum_{n \geq 0} \frac{t^n}{n!} \sum_{(\sg,w) \in C_k \wr S_n}
x^{\Umch{(\sg,w)}} = \frac{(k-1)(1-x)}{(k-1)(1-x) + k(e^{(k-1)(x-1)t}-1)}.
\end{equation}

In fact, all of the generating functions
(\ref{ris})--(\ref{Up}) are special cases of more refined generating
functions for $C_k \wr S_n$ where we keep track of more statistics.
For $\Upsilon \subseteq C_k \wr S_j$, we shall consider
generating functions of the form
\begin{equation}\label{defDU}
D_k^{\Upsilon}(x,p,q,r,t) = \sum_{n \geq 0} \frac{t^n}{[n]_{p,q}!}
\sum_{(\sg,w) \in C_k \wr S_n} q^{\inv{\sg}}p^{\coinv{\sg}}r^{||w||}
x^{\Umch{(\sg,w)}}
\end{equation}
where $||w||=||w_1\ldots w_n||=w_1+\cdots+w_n$, and $\inv{\sg}$
(resp. $\coinv{\sg}$) is the number of inversions (resp.
coinversions) of $\sg$ defined for $\sg=\sg_1\ldots\sg_n$ as the number of pairs $i<j$ such that $\sg_i>\sg_j$ (resp. $\sg_i<\sg_j$). Let
\begin{eqnarray*}
\Upsilon_{{\bf r}} &=& \{(1~2,0~0),(1~2,0~1)\},\\
\Upsilon_{{\bf w}} &=&  \{(1~2,0~0)\},\\
\Upsilon_{{\bf s}} &=& \{(1~2,0~1)\}, \ \mbox{and} \\
\Upsilon_{{\bf d}} &=& \{(1~2,0~1),(1~2,1~0)\}.
\end{eqnarray*}
Thus $\Upsilon_{{\bf r}}$-matches correspond to rises,
$\Upsilon_{{\bf w}}$-matches correspond to weak rises, and
$\Upsilon_{{\bf s}}$-matches correspond to strict rises. We shall
find $D_k^{\Upsilon_{{\bf a}}}(x,p,q,r,t)$ for ${\bf a} \in \{{\bf
r},{\bf w}, {\bf s}\}$ and find $D_k^{\Upsilon_{{\bf
d}}}(x,p,q,1,t)$. For example, define the $p,q$-analogues of $n$,
$n!$, $\binom{n}{k}$, and $\binom{n}{a_1, \ldots,a_m}$ by
\begin{eqnarray*}
\ [n]_{p,q} &=&  \frac{p^n - q^n}{p-q} = p^{n-1} + p^{n-2} q  +\cdots + pq^{n-2} +q^{n-1}, \\
\ [n]_{p,q}! &=& [n]_{p,q} [n-1]_{p,q} \cdots [2]_{p,q} [1]_{p,q}, \\
\ \qbin{n}{k}{p,q} &=& \frac{[n]_{p,q}!}{[k]_{p,q}![n-k]_{p,q}!}, \ \mbox{and} \\
\ \qbin{n}{a_1, \ldots, a_m}{p,q} &=& \frac{[n]_{p,q}!}{[a_1]_{p,q}! \cdots [a_m]_{p,q}!},
\end{eqnarray*}
respectively.
We define the $q$-analogues of $n$, $n!$, $\binom{n}{k}$, and $\binom{n}{a_1, \ldots,a_m}$ by $[n]_{1,q}$, $[n]_{1,q}!$, $\qbin{n}{k}{1,q}$, and
$\qbin{n}{a_1,\ldots,a_m}{1,q}$, respectively.
Then we will prove that
\begin{eqnarray}\label{pqris}
&&D_k^{\Upsilon_{{\bf r}}}(x,p,q,r,t) = \sum_{n \geq 0} \frac{t^n}{[n]_{p,q}!} \sum_{(\sg,w) \in C_k \wr S_n}
q^{\inv{\sg}}p^{\coinv{\sg}}r^{||w||}x^{\ris{(\sg,w)}} = \nonumber \\
&&\frac{1-x}{1-x + \sum_{n \geq 1}
\frac{p^{\binom{n}{2}}((x-1)t)^n}{[n]_{p,q}!}
\rbinom{n+k-1}{n}}
\end{eqnarray}
which reduces to (\ref{ris}) when we set $p=q=r=1$.

We shall prove our formulas for the generating functions $D_k^{\Upsilon_{{\bf a}}}(x,p,q,r,t)$ for ${\bf a}
\in \{{\bf r},{\bf w}, {\bf s}\}$ and $D_k^{\Upsilon_{{\bf d}}}(x,p,q,1,t)$
by applying a ring
homomorphism, defined on the ring $\Lambda$ of symmetric functions
over infinitely many variables  $x_1,x_2, \ldots$, to a simple
symmetric function identity. There has been a long line of research,
\cite{b}, \cite{b2}, \cite{br}, \cite{L}, \cite{LR}, \cite{MenRem1},
\cite{Book}, \cite{RRW}, \cite{Wag}, \cite{Men}, which shows that a
large number of generating functions for permutation statistics can
be obtained by applying homomorphisms defined on the ring of
symmetric functions $\Lambda$ over infinitely many variables
$x_1,x_2, \ldots $ to simple symmetric function identities.  For
example, the $n$-th elementary symmetric function, $e_n$, and the
$n$-th homogeneous  symmetric function, $h_n$, are defined by the
generating functions
\begin{equation}\label{eq:E}
E(t) =  \sum_{n \geq 0} e_n t^n = \prod_i (1+x_i t)
\end{equation}
and
\begin{equation}\label{eq:H}
H(t) =  \sum_{n \geq 0} h_n t^n = \prod_i \frac{1}{1-x_i t}.
\end{equation}
We let $P(t) = \sum_{n \geq 0} p_n t^n$ where $p_n = \sum_i x_i^n$
is the $n$-th power symmetric function. A partition of
$n$ is a sequence  $\mu =
(\mu_1, \ldots, \mu_k)$ such that $0 < \mu_1 \leq \cdots \leq \mu_k$ and
$\mu_1 + \cdots + \mu_k = n$.  We write $\mu \vdash n$ if $\mu$  is partition
of $n$ and we let $\ell(\mu)$ denote the number of parts of $\mu$.
If $\mu \vdash n$, we set $h_\mu = \prod_{i=1}^{\ell(\mu)}
h_{\mu_i}$,  $e_\mu = \prod_{i=1}^{\ell(\mu)} e_{\mu_i}$,
 and $p_\mu = \prod_{i=1}^{\ell(\mu)} p_{\mu_i}$. Let
$\Lambda_n$ denote the space of homogeneous symmetric functions of
degree $n$ over infinitely many variables $x_1,x_2, \ldots$ so that
$\Lambda=\oplus_{n\geq 0} \Lambda_n$. It is well know that
$\{e_\lambda:\lambda \vdash n\}$, $\{h_\lambda:\lambda \vdash n\}$,
and $\{p_\lambda:\lambda \vdash n\}$ are all bases of $\Lambda_n$.
It follows that $\{e_0, e_1, \ldots \}$ is an algebraically
independent set of generators for $\Lambda$ and hence we can define
a ring homomorphism $\xi:\Lambda \rightarrow R$ where $R$ is a ring
by simply specifying $\xi(e_n)$ for all $n \geq 0$.

Now it is well-known that
\begin{equation}\label{HEI}
H(t) =  1/E(-t)
\end{equation}
and
\begin{equation}\label{PEI}
P(t) = \frac{\sum_{n \geq 1} (-1)^{n-1} n e_n t^n}{E(-t)}.
\end{equation}
A surprisingly large number of results on generating functions for
various permutation statistics in the literature and large number of
new generating functions can be derived by applying homomorphisms on
$\Lambda$ to simple identities such as
 (\ref{HEI}) and (\ref{PEI}). We shall show that all our
generating functions arise by applying appropriate ring homomorphisms
to identity (\ref{HEI}). For example, we shall show that (\ref{pqris})
arises by applying the ring homomorphism  $\xi$ to identity (\ref{HEI})
where
$\xi(e_0) =1$ and
$$\xi(e_n) = \frac{(-1)^{n-1}(x-1)^{n-1}p^{\binom{n}{2}}}{[n]_{p,q}!}
\rbinom{n+k-1}{n}
$$ for all $n \geq 0$.

We can  use our  formulas for the generating functions
$D_k^{\Upsilon_{{\bf a}}}(x,p,q,r,t)$ for ${\bf a} \in \{{\bf
r},{\bf w}, {\bf s}\}$ and $D_k^{\Upsilon_{{\bf d}}}(x,p,q,1,t)$,
to derive a number of other generating functions. For example, for
any $\Upsilon \subseteq C_k \wr S_j$ such that $red(u) =u$ for all
$(\tau,u) \in \Upsilon$, let
\begin{equation}\label{AU}
A_k^{\Upsilon}(p,q,r,t) = \sum_{n \geq 0} \frac{t^n}{[n]_{p,q}!}
\sum_{(\sg,w) \in C_k \wr S_n} q^{\inv{\sg}}p^{\coinv{\sg}}r^{||w||}
\chi( \Umch{(\sg,w)}=0)
\end{equation}
where for any statement $B$, we let $\chi(B)$ equal 1 if $B$ is true
and equal 0 if $B$ is false. Thus $A_k^\Upsilon(p,q,r,t)$ is the
generating function counting elements of $C_k \wr S_n$ with no
$\Upsilon$-matches. We shall prove that if
\begin{equation}\label{defNU}
N_k^{\Upsilon}(x,p,q,r,t) = \sum_{n \geq 0} \frac{t^n}{[n]_{p,q}!}
\sum_{(\sg,w) \in C_k \wr S_n} q^{\inv{\sg}}p^{\coinv{\sg}}r^{||w||}
x^{\Ulap{(\sg,w)}},
\end{equation}
then
\begin{equation}\label{NU}
N_k^\Upsilon(x,p,q,r,t) = \frac{A_k^\Upsilon(p,q,r,t)}{1 -x(([k]_rt-1)
A_k^\Upsilon(p,q,r,t) +1)}.
\end{equation}
This result is an analogue of a result by Kitaev~\cite{K} for
permutations. Since our generating functions will allow us to derive
expressions for $A_k^{\Upsilon_{{\bf a}}}(p,q,r,t)$ for ${\bf a} \in
\{{\bf r},{\bf w},{\bf s}\}$,
 we will automatically be able to find
the generating functions for the distribution of non-overlapping
$\Upsilon_{{\bf a}}$-matches for
${\bf a} \in \{{\bf r},{\bf w},{\bf s}\}$.
There are two additional generating functions that we can obtain
in each of our examples. For example, it
is easy to see that since $\Upsilon_{{\bf r}}$-matches correspond to rises,
then the coefficient of $x$ in
$N_k^{\Upsilon_{{\bf r}}}(x,p,q,r,t) - D_k^{\Upsilon_{{\bf r}}}(x,p,q,r,t)$,
written
$$(N_k^{\Upsilon_{{\bf r}}}(x,p,q,r,t) -D_k^{\Upsilon_{{\bf r}}}(x,p,q,r,t))|_x,$$
is the generating function for $(\sg,w) \in C_k \wr S_n$ such
that $(\sg,w)$ has exactly 2 rises which overlap, i.e. there
is exactly one pattern match of
$$\Upsilon =
\{(0~1~2,0~0~0), (0~1~2,0~0~1), (0~1~2,0~1~1), (0~1~2,0~1~2)\}$$
and no other rises. Moreover,
$$D_k^{\Upsilon_{{\bf r}}}(x,p,q,r,t)|_{x^2}-
[N_k^{\Upsilon_{{\bf r}}}(x,p,q,r,t) -D_k^{\Upsilon_{{\bf r}}}(x,p,q,r,t)|_x]$$
is the generating function for $(\sg,w) \in C_k \wr S_n$ such
that $(\sg,w)$ has exactly 2 rises which do not overlap.  Our
results will allow us to find explicit formulas for these two additional
types of generating functions for rises, weak rises, and strict rises.

The outline of this paper is as follows.
In section 2, we shall provide the necessary background in
symmetric functions that we shall need to derive our generating functions.
In section 3, we shall give our proofs of the generating functions
$D_k^{\Upsilon_{{\bf a}}}(x,p,q,r,t)$ for
${\bf a} \in \{{\bf r},{\bf w},{\bf s}\}$
 and $D_k^{\Upsilon_{{\bf d}}}(x,p,q,1,t)$. Finally,
in sections 4 and 5, we shall find explicit expressions for
$$(N_k^{\Upsilon_{{\bf a}}}(x,p,q,r,t) -D_k^{\Upsilon_{{\bf a}}}(x,p,q,r,t))|_x$$
and
$$D_k^{\Upsilon_{{\bf a}}}(x,p,q,r,t)|_{x^2}-
[N_k^{\Upsilon_{{\bf a}}}(x,p,q,r,t) -D_k^{\Upsilon_{{\bf
a}}}(x,p,q,r,t)|_x]$$ for ${\bf a} \in \{{\bf r},{\bf w},{\bf
s}\}$. In section 6, we shall give tables of the number of
various types of permutations $(\sg,w) \in C_k \wr S_n$ that can be
computed from our generating functions for small values of $k$ and
$n$. We shall see that various sequences associated with our sets of
permutations appear in OEIS~\cite{OEIS} and hence our sequences
count other combinatorial objects. Moreover, we shall see that
for fixed $n$, some
of the sequences are generated by certain natural polynomials in
$k$.  For example, we let $A^{\Upsilon}_{n,k}$
denote the number of $(\sg,w) \in
C_k \wr S_n$ for which $\Umch{(\sg,w)}=0$. We shall show
that if $\Upsilon = \{(1~2,0~0)\}$, then
 $A^{\Upsilon}_{n,k} = \sum_{j=1}^n (-1)^{n-j} j!S_{n,j}k^j$
for all $k \geq 2$, where $S_{n,k}$ is the Stirling number of the
second kind.  Similarly, if $\Upsilon = \{(1~2,0~1), (1~2,1~0)\}$,
then  $A^{\Upsilon}_{n,k}$ is an Eulerian polynomial. That is, for
all $k \geq 2$, $A^{\Upsilon}_{n,k} = \sum_{\sg \in S_n}
k^{\des{\sg}+1}$ for where $\des{\sg}$ is the number of descents of
$\sg$. The connections to the Stirling numbers of the second kind
and to the Eulerian polynomials were observed by Einar
Steingr\'imsson and we prove this in this paper. Finally, in section
7, we shall state a few problems for further research.

\section{Symmetric Functions}

In this section we give the necessary background on symmetric
functions needed for our proofs of the generating functions.

Let $\Lambda$ denote the ring of symmetric functions over infinitely
many variables $x_1, x_2, \ldots $ with coefficients in the field of
complex numbers $\mathbb{C}$. The $n^{\text{th}}$ elementary
symmetric function $e_n$ in the variables $x_1,x_2,\dots$ is given
by
\begin{equation*}
E(t) =  \sum_{n \geq 0} e_n t^n = \prod_i (1+x_i t)
\end{equation*}
and the $n^{\text{th}}$ homogeneous symmetric function $h_n$ in the variables $x_1,x_2,\dots$ is
given by
\begin{equation*}
H(t) =  \sum_{n \geq 0} h_n t^n = \prod_i \frac{1}{1-x_i t}.
\end{equation*}
Thus
\begin{equation}
\label{HE}
H(t) =  1/E(-t).
\end{equation}
Let $\la = (\la_1,\dots,\la_\ell)$ be an integer partition, that is,
$\la$ is a finite sequence of weakly increasing positive
integers.  Let $\ell(\la) =l$ denote the number of parts of
$\la$. If the sum of these integers is $n$, we say that $\la$ is a
partition of $n$ and write $\la \vdash n$.  For any partition $\la =
(\la_1,\dots,\la_\ell)$, let $e_\la = e_{\la_1} \cdots
e_{\la_\ell}$. The well-known fundamental theorem of symmetric
functions says that $\{e_\la : \text{$\la$ is a partition}\}$ is a
basis for $\La$ or that $\{e_0,e_1,\ldots \}$ is an algebraically
independent set of generators for $\Lambda$.
Similarly, if we define  $h_\la = h_{\la_1} \cdots
h_{\la_\ell}$, then $\{h_\la : \text{$\la$ is a partition}\}$ is
also a basis for $\La$. Since $\{e_0,e_1, \ldots \}$ is an algebraically independent
set of generators for $\Lambda$, we can specify a ring homomorphism
$\theta$ on $\Lambda$ by simply defining $\theta(e_n)$ for all
$n \geq 0$.

Since the elementary symmetric functions $e_\lambda$ and the
homogeneous symmetric functions $h_\lambda$ are both bases for
$\Lambda$, it makes sense to talk about the coefficient of the
homogeneous symmetric functions when written in terms of the
elementary symmetric function basis. These coefficients has been shown
to equal the sizes of a certain sets of combinatorial objects up to a
sign. A {\em brick tabloid} of shape $(n)$ and type $\la
=(\la_1,\ldots, \la_k)$ is a filling of a row of $n$ squares of
cells with bricks of lengths $\la_1, \ldots, \la_k$ such that bricks
do not overlap. One brick tabloid of shape $(12)$ and type
$(1,1,2,3,5)$ is displayed below.

\fig{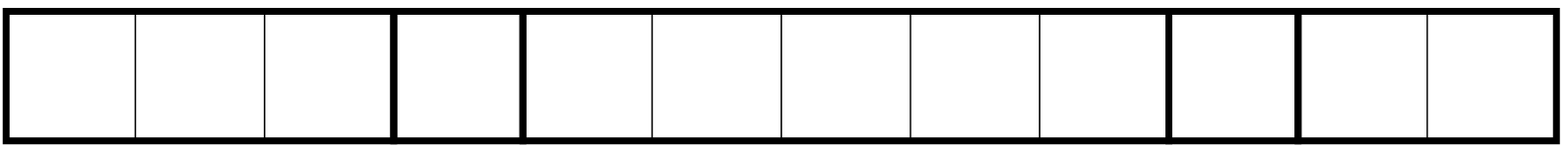}{A brick tabloid of shape $(12)$ and type $(1,1,2,3,5)$.}

Let $\mathcal{B}_{\la,n}$ denote the set of all $\la$-brick tabloids of shape $(n)$ and let
$B_{\la,n} =|\mathcal{B}_{\la,n}|$.  Through simple recursions
stemming from \eqref{HE}, E\u{g}ecio\u{g}lu and Remmel proved in
\cite{ER} that
\begin{equation}
\label{omar}
h_n = \sum_{\la \vdash n} (-1)^{n - \ell(\la)} B_{\la,n} e_\la.
\end{equation}

We end this section with two lemmas that will be needed in later
sections. Both of the lemmas follow from simple codings of a basic
result of Carlitz \cite{Car2} that
$$
\qbinom{n}{k} = \sum_{r \in \mathcal{R}(1^k0^{n-k})} q^{inv(r)},
$$
where $\mathcal{R}(1^k0^{n-k})$ is the number of rearrangements of
$k$ 1's and $n-k$ 0's.
We start with a lemma from \cite{MRR}.
Fix a brick tabloid $T = (b_1, \ldots,b_{\ell(\mu)}) \in \mathcal{B}_{\mu,n}$.  Let $IF(T)$
denote the set of all fillings of the cells of
$T = (b_1, \ldots, b_{\ell(\mu)})$ with the numbers
$1, \ldots, n$ so that the numbers increase within each brick
reading from left to right. We then think of each such filling as a
permutation of $S_n$ by reading the numbers from left to right in each row. For example,
Figure \ref{figure:fil1} pictures an
element of $IF(3,6,3)$ whose corresponding permutation
is $4~6~12~1~5~7~8~10~11~2~3~9$.

\fig{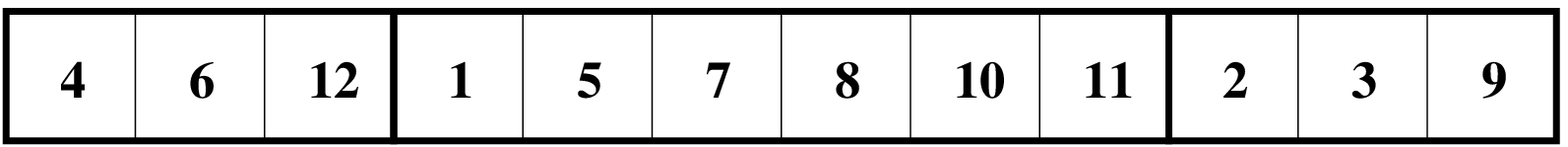}{An element of $IF(3,6,3)$.}

Then the following lemma which is proved in \cite{MRR}
gives a combinatorial interpretation to
${p}^{\sum_{i=1}^{\ell(\mu)} \binom{b_i}{2}} \pqbinom{n}{b_1, \ldots, b_{\ell(\mu)}}$.
\begin{lemma}
\label{Carlitz}
If $T=(b_1, \ldots, b_{\ell(\mu)})$ is a brick tabloid in $\mathcal{B}_{\mu,n}$, then
\begin{equation*}
p^{\sum_{i=1}^{\ell(\mu)} \binom{b_i}{2}} \pqbinom{n}{b_1, \ldots,
b_{\ell(\mu)}} = \sum_{\sg \in IF(T)} q^{inv(\sg)} p^{coinv(\sg)}.
\end{equation*}
\end{lemma}

Another well-known combinatorial interpretation for
$\qbinom{n+k-1}{k-1}$ is that it is equal to the sum of the sizes of
the partitions that are contained in an $n \times (k-1)$ rectangle.
Thus we have the following lemma.
\begin{lemma}
\label{Carlitz3}
\begin{equation*}
\sum_{0 \leq a_1 \leq \cdots \leq a_n \leq k-1} q^{a_1+ \cdots + a_n} =
\qbinom{n+k-1}{n}.
\end{equation*}
\end{lemma}

\section{Generating Functions}

The main goal of this section is to prove generating functions that
specialize to the generating functions (\ref{ris})--(\ref{Up}) given
in the introduction.

We start by proving a generating function which specializes
to (\ref{ris}).

\begin{theorem}\label{thm:pqris} Let $\Upsilon_{{\bf r}} =\{(1~2,0~0),(1~2,0~1)\}$.
For all $k \geq 2$,
\begin{eqnarray}\label{eq:pqris}
D_k^{\Upsilon_{{\bf r}}}(x,p,q,r,t) &=& \sum_{n \geq 0} \frac{t^n}{[n]_{p,q}!} \sum_{(\sg,w) \in C_k \wr S_n}
q^{\inv{\sg}}p^{\coinv{\sg}}r^{||w||}x^{\ris{(\sg,w)}} \nonumber \\
&=&
\frac{1-x}{1-x + \sum_{n \geq 1} \frac{p^{\binom{n}{2}}((x-1)t)^n}{[n]_{p,q}!}
\rbinom{n+k-1}{n}}.
\end{eqnarray}
\end{theorem}
 \begin{proof}
Define a ring homomorphism
$\Gamma:\Lambda \rightarrow \mathbb{Q}(p,q,r,x)$ by setting $\Gamma(e_0) =1$
and
\begin{equation}
\Gamma(e_n) = (-1)^{n-1}(x-1)^{n-1}\frac{\rbinom{n+k-1}{n}}{[n]_{p,q}!}
p^{\binom{n}{2}}
\end{equation}
for $n \geq 1$. Then we claim that
\begin{equation}\label{ris1}
[n]_{p,q}!\Gamma(h_n) = \sum_{(\sg,w) \in
C_k \wr S_{n}}
q^{\inv{\sg}}p^{\coinv{\sg}} r^{||w||} x^{\ris{(\sg,w)}}
\end{equation}
for all $n \geq 1$.
That is,
\begin{eqnarray}\label{ris2}
&&\ [n]_{p,q}!\Gamma(h_{n}) = \nonumber \\
&& \ [n]_{p,q}! \sum_{\mu\vdash n} (-1)^{n-\ell(\mu)}B_{\mu, (n)}\Gamma(e_{\mu}) = \nonumber \\
&& \ [n]_{p,q}! \sum_{\mu \vdash n} (-1)^{n-\ell(\mu)} \sum_{(b_1,
\ldots, b_{\ell(\mu)}) \in \mathcal{B}_{\mu,n}}
\prod_{j=1}^{\ell(\mu)} (-1)^{b_j-1}(x-1)^{b_j-1}
\frac{\rbinom{b_j+k-1}{b_j}}{[b_j]_{p,q}!} p^{\binom{b_j}{2}}= \nonumber \\
&& \ \sum_{\mu \vdash n} \sum_{(b_1, \ldots, b_{\ell(\mu)}) \in \mathcal{B}_{\mu,n}} p^{\sum_{j=1}^{\ell(\mu)} \binom{b_j}{2}}\pqbinom{n}{b_1,\ldots,b_{\ell(\mu)}} \prod_{j=1}^{\ell(\mu)} (x-1)^{b_j-1}
\rbinom{b_j+k-1}{b_j}.
\end{eqnarray}

Next we want to give a combinatorial interpretation to
(\ref{ris2}). By Lemma \ref{Carlitz}, for each brick tabloid $T=
(b_1, \ldots, b_{\ell(\mu)})$, we can interpret
$p^{\sum_{j=1}^{\ell(\mu)} \binom{b_j}{2}}
\pqbinom{n}{b_1,\ldots,b_{\ell(\mu)}}$ as the sum of the weights
of all fillings of $T$ with a permutation $\sg \in S_{n}$ such
that $\sg$ is increasing in each brick and we weight $\sg$ with
$q^{\inv{\sg}}p^{\coinv{\sg}}$.  By Lemma \ref{Carlitz3}, we can
interpret the term $\prod_{j=1}^{\ell(\mu)}\rbinom{b_j+k-1}{b_j}$
as the sum of the weights of  fillings  $w= w_1 \ldots w_{n}$
where the elements of $w$ are between 0 and $k-1$ and are weakly
increasing in each brick and where we weight $w$ by $r^{w_1+
\cdots +w_{n}}$.  Finally, we interpret $\prod_{j=1}^{\ell(\mu)}
(x-1)^{b_j-1}$ as all ways of picking a label of the cells of each
brick except the final cell with either an $x$ or a $-1$. For
completeness, we label the final cell of each brick with $1$. We
shall call all such objects created in this way filled labelled
brick tabloids and let $\mathcal{F}_{n}$ denote the set of all
filled labelled brick tabloids that arise in this way.  Thus a $C
\in \mathcal{F}_{n}$ consists of a brick tabloid $T$, a
permutation $\sg \in S_{n}$, a sequence $w \in
\{0,\ldots,k-1\}^{n}$, and a labelling $L$ of the cells of $T$
with elements from $\{x,1,-1\}$ such that
\begin{enumerate}
\item $\sg$ is strictly increasing in each brick, \item $w$ is
weakly increasing in each brick, \item the final cell of each
brick is labelled with 1, and \item each cell which is not a final
cell of a brick is labelled with x or $-1$.
\end{enumerate}
We then define the weight $w(C)$ of $C$ to be
$q^{\inv{\sg}}p^{\coinv{\sg}} r^{||w||}$ times the product of all
the $x$ labels in $L$ and the sign $sgn(C)$ of $C$ to be
the product of all the $-1$ labels in $L$. For example,
if $n =12$, $k=4$, and $T =(4,3,3,2)$, then Figure \ref{figure:ris1}
pictures such a composite object $C \in \mathcal{F}_{12}$ where
$w(C) = q^{24}p^{32}r^{17}x^5$ and $sgn(C) =-1$.

Thus
\begin{equation}\label{ris4}
[n]_{p,q}!\Gamma(h_{n}) = \sum_{C \in \mathcal{F}_{n}}
sgn(C) w(C).
\end{equation}

\fig{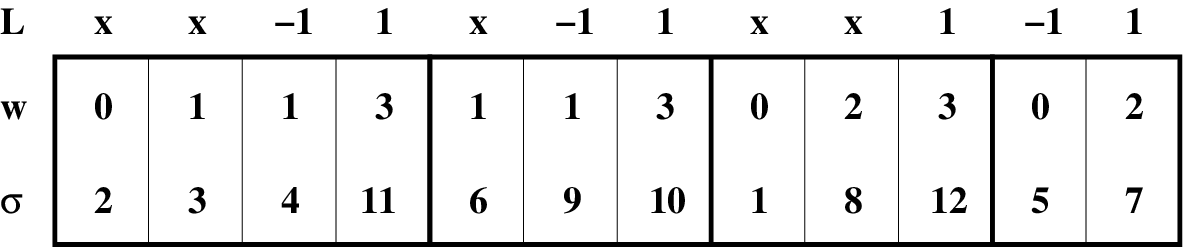}{A composite object $C \in \mathcal{F}_{12}$.}

Next we define a weight-preserving sign-reversing involution
$I:\mathcal{F}_{n} \rightarrow \mathcal{F}_{n}$.  To define
$I(C)$, we scan the cells of $C =(T,\sg,w,L)$ from left  to right
looking for the leftmost cell $t$ such that either (i) $t$ is
labelled with $-1$ or (ii) $t$ is at the end of a brick $b_j$ and
the brick $b_{j+1}$ immediately following $b_j$ has the property
that $\sg$ is strictly increasing in all the cells corresponding
to $b_j$ and $b_{j+1}$ and $w$ is weakly  increasing in all the
cells corresponding to $b_j$ and $b_{j+1}$.  In case (i), $I(C)
=(T',\sg',w',L')$ where $T'$ is the result of  replacing the brick
$b$ in $T$ containing $t$ by two bricks $b^*$ and $b^{**}$ where
$b^*$ contains the cell $t$ plus all the cells in $b$ to the left
of $t$ and $b^{**}$ contains all the cells of $b$ to the right of
$t$, $\sg =\sg'$, $w = w'$, and $L'$ is the labelling that results
from $L$ by changing the label of cell $t$ from $-1$ to $1$. In
case (ii), $I(C) =(T',\sg',r',L')$ where $T'$ is the result of
replacing the bricks $b_j$ and $b_{j+1}$ in $T$ by a single brick
$b,$ $\sg =\sg'$, $w = w'$, and $L'$ is the labelling that results
from $L$ by changing the label of cell $t$ from $1$ to $-1$. If
neither case (i) or case (ii) applies, then we let $I(C) =C$. For
example, if $C$ is the element of $\mathcal{F}_{12}$ pictured in
Figure \ref{figure:ris1}, then $I(C)$ is pictured in Figure
\ref{figure:ris2}.

\fig{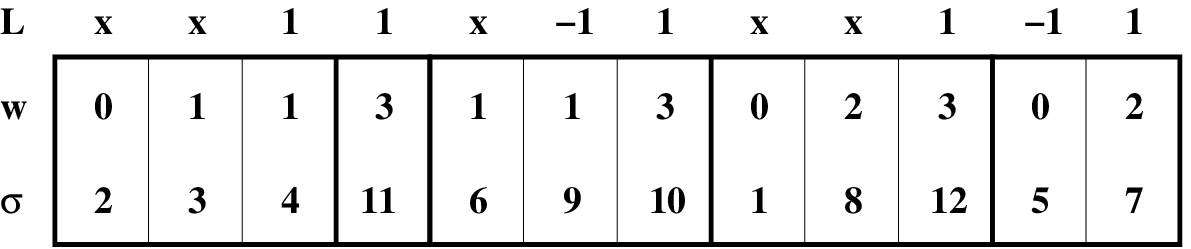}{$I(C)$ for $C$ in Figure \ref{figure:ris1}.}

It is easy to see that $I$ is a weight-preserving sign-reversing
involution and hence $I$ shows that
\begin{equation}\label{ris5}
[n]_{p,q}!\Gamma(h_n) = \sum_{C \in \mathcal{F}_{n},I(C) = C}
sgn(C) w(C).
\end{equation}

Thus we must examine the fixed points $C = (T,\sg,w,L)$ of $I$.
First there can be no $-1$ labels in $L$ so that $sgn(C) =1$.
Moreover,  if $b_j$ and $b_{j+1}$ are two consecutive bricks in $T$
and $t$ is the last cell of $b_j$, then it can not be the case that
$\sg_{t} < \sg_{t+1}$ and $w_t \leq w_{t+1}$ since otherwise we
could combine $b_j$ and $b_{j+1}$. For any such fixed point, we
associate an element $(\sg,w) \in C_k \wr S_{n}$. For example, a
fixed point of $I$ is pictured in Figure \ref{figure:ris3} where
\begin{eqnarray*}
\sg &=& 2~3~4~6~9~10~11~1~8~12~5~7 \ \mbox{and} \\
w &=& 0~1~1~3~1~1~3~0~2~3~3~3.
\end{eqnarray*}
It follows that if cell $t$ is at the end of a brick, then $t \not
\in Ris((\sg,\ep))$. However if $v$ is a cell which is not at the
end of a brick, then our definitions force $\sg_{v} < \sg_{v+1}$
and $w_v \leq w_{v+1}$  so that $v \in Ris((\sg,\ep))$. Since each
such cell $v$ must be labelled with an $x$, it follows that
$sgn(C)w(C) =
q^{inv(\sg)}p^{coinv(\sg)}r^{||w||}x^{\ris{(\sg,\ep)}}$.
  Vice versa, if
$(\sg,w) \in C_k \wr S_{n}$, then we can create a fixed point $C
=(T,\sg,w,L)$ by having the bricks in $T$ end at cells of the form
$t$ where $t \not \in Ris((\sg,\ep))$, and labelling each cell $t
\in Ris((\sg,\ep))$ with $x$ and labelling the remaining cells
with $1$. Thus we have shown that
\begin{equation*}
[n]_{p,q}!\Gamma(h_n) = \sum_{(\sg,w) \in C_k \wr S_{n}}
q^{\inv{\sg}}p^{\coinv{\sg}} r^{||w||} x^{\ris{(\sg,w)}}
\end{equation*}
as desired.

\fig{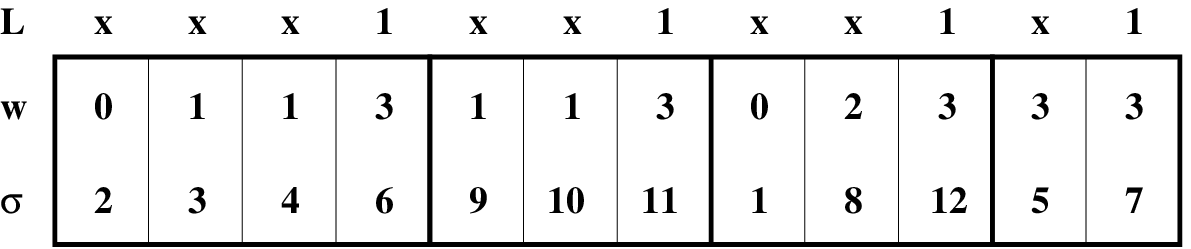}{A fixed point of $I$.}

Applying $\Gamma$ to the identity $H(t) = (E(-t))^{-1}$, we get
\begin{eqnarray*}
\sum_{n \geq 0} \Gamma(h_n) t^n &=& \sum_{n\geq 0} \frac{t^{n}}{[n]_{p,q}!} \sum_{(\sg,w)\in C_k\wr S_{n}} q^{\inv{\sg}}p^{\coinv{\sg}}
r^{||w||}x^{\ris{(\sg,w)}} \\
&=& \frac{1}{1+\sum_{n\geq 1} (-t)^n\Gamma(e_n)} \\
&=& \frac{1}{1+\sum_{n\geq 1}(-1)^{n} t^{n}
\frac{(-1)^{n-1}(x-1)^{n-1}p^{\binom{n}{2}}}{[n]_{p,q}!}\rbinom{n+k-1}{k-1}} \\
&=& \frac{1-x}{1-x + \sum_{n\geq
1}\frac{p^{\binom{n}{2}}(x-1)^nt^{n}}{[n]_{p,q}!}\rbinom{n+k-1}{k-1}}
\end{eqnarray*}
which proves (\ref{eq:pqris}).
\end{proof}

In essentially the same way, we can prove a result which
specializes to (\ref{wris}).

\begin{theorem}\label{thm:pqwris} Let $\Upsilon_{{\bf w}}= \{(1~2,0~0)\}$.
Then for all $k \geq 2$,
\begin{eqnarray}\label{eq:pqwris}
D_k^{\Upsilon_{{\bf w}}}(x,p,q,r,t) &=&
\sum_{n \geq 0} \frac{t^n}{[n]_{p,q}!} \sum_{(\sg,w) \in C_k \wr S_n}
q^{\inv{\sg}}p^{\coinv{\sg}}r^{||w||}x^{\wris{(\sg,w)}}
\nonumber \\
&=&
\frac{1-x}{1-x + \sum_{n \geq 1} \frac{p^{\binom{n}{2}}((x-1)t)^n}{[n]_{p,q}!}
[k]_{r^n}}.
\end{eqnarray}
\end{theorem}
 \begin{proof}
Define a ring homomorphism
$\Gamma_w:\Lambda \rightarrow \mathbb{Q}(p,q,r,x)$ by setting
$\Gamma_w(e_0)=1$
and
\begin{equation}
\Gamma_w(e_n) = (-1)^{n-1}(x-1)^{n-1}\frac{[k]_{r^n}}{[n]_{p,q}!}
p^{\binom{n}{2}}
\end{equation}
for $n \geq 1$. Then we claim that
\begin{equation}\label{wris1}
[n]_{p,q}!\Gamma_w(h_n) = \sum_{(\sg,w) \in
C_k \wr S_{n}}
q^{\inv{\sg}}p^{\coinv{\sg}} r^{||w||} x^{\wris{(\sg,w)}}
\end{equation}
for all $n \geq 1$.
That is,
\begin{eqnarray}\label{wris2}
&&\ [n]_{p,q}!\Gamma_w(h_{n}) = \nonumber \\
&& \ [n]_{p,q}! \sum_{\mu\vdash n} (-1)^{n-\ell(\mu)}B_{\mu, (n)}\Gamma_w(e_{\mu}) = \nonumber \\
&& \ [n]_{p,q}! \sum_{\mu \vdash n} (-1)^{n-\ell(\mu)} \sum_{(b_1,
\ldots, b_{\ell(\mu)}) \in \mathcal{B}_{\mu,n}}
\prod_{j=1}^{\ell(\mu)} (-1)^{b_j-1}(x-1)^{b_j-1}
\frac{[k]_{r^{b_j}}}{[b_j]_{p,q}!} p^{\binom{b_j}{2}}= \nonumber \\
&& \ \sum_{\mu \vdash n} \sum_{(b_1, \ldots, b_{\ell(\mu)}) \in \mathcal{B}_{\mu,n}} p^{\sum_{j=1}^{\ell(\mu)} \binom{b_j}{2}}\pqbinom{n}{b_1,\ldots,b_{\ell(\mu)}} \prod_{j=1}^{\ell(\mu)} (x-1)^{b_j-1}
[k]_{r^{b_j}}.
\end{eqnarray}

Next we want to give a combinatorial interpretation to (\ref{wris2}).
By Lemma \ref{Carlitz} for each brick tabloid
$T= (b_1, \ldots, b_{\ell(\mu)})$, we can interpret
$p^{\sum_{j=1}^{\ell(\mu)} \binom{b_j}{2}}
\pqbinom{n}{b_1,\ldots,b_{\ell(\mu)}}$ as the sum of the weights of all fillings of $T$ with
a permutation $\sg \in S_{n}$ such that $\sg$ is increasing in each brick
and we weight $\sg$ by $q^{\inv{\sg}}p^{\coinv{\sg}}$.
For each $j$, we have a factor
$$[k]_{r^{b_j}} =
r^{0\cdot b_j} +r^{1\cdot b_j}+ \cdots  +r^{(k-1)\cdot b_j}.$$ We
shall interpret the term $r^{sb_j}$ as indicating that we will
fill the top of each cell of a brick $b_j$ with $s$. Thus we can
interpret $\prod_{j=1}^{\ell(\mu)}[k]_{r^{b_j}}$ as filling of the
brick with a sequence $w_1 \ldots w_n \in [k]^n$ such that $w$ is
constant in each brick and where we weight $w$ by $r^{||w||}$.
Finally, we interpret $\prod_{j=1}^{\ell(\mu)} (x-1)^{b_j-1}$ as
all ways of picking a label of the cells of each brick except the
final cell with either an $x$ or a $-1$. For completeness, we
label the final cell of each brick with $1$. We shall call all
such objects created in this way filled labelled brick tabloids
and let $\mathcal{G}_{n}$ denote the set of all filled labelled
brick tabloids that arise in this way.  Thus a $C \in
\mathcal{G}_{n}$ consists of a brick tabloid $T$, a permutation
$\sg \in S_{n}$, a sequence $w \in \{0,\ldots,k-1\}^{n}$, and a
labelling $L$ of the cells of $T$ with elements from $\{x,1,-1\}$
such that
\begin{enumerate}
\item $\sg$ is strictly increasing in each brick, \item $w$ is
constant in each brick, \item the final cell of each brick is
labelled with 1, and \item each cell which is not a final cell of
a brick is labelled with x or $-1$.
\end{enumerate}
We then define the weight $w(C)$ of $C$ to be
$q^{\inv{\sg}}p^{\coinv{\sg}} r^{||w||}$ times the product of all
the $x$ labels in $L$ and the sign $sgn(C)$ of $C$ to be
the product of all the $-1$ labels in $L$. For example,
if $n =12$, $k=4$, and $T =(4,3,3,2)$, then Figure \ref{figure:wris1}
pictures such a composite object $C \in \mathcal{G}_{12}$ where
$w(C) = q^{24}p^{42}r^{19}x^5$ and $sgn(C) =-1$.

Thus
\begin{equation}\label{wris4}
[n]_{p,q}!\Gamma_w(h_{n}) = \sum_{C \in \mathcal{G}_{n}}
sgn(C) w(C).
\end{equation}

\fig{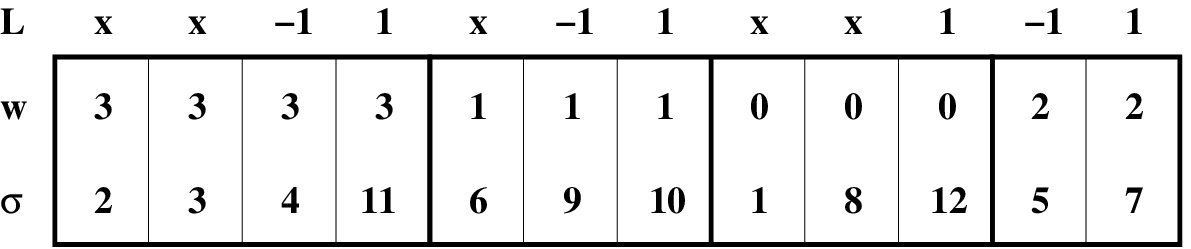}{A composite object $C \in \mathcal{G}_{12}$.}

Next we define a weight-preserving sign-reversing involution
$I_w:\mathcal{G}_{n} \rightarrow \mathcal{G}_{n}$.  To define
$I_w(C)$, we scan the cells of $C =(T,\sg,w,L)$ from left to right
looking for the leftmost cell $t$ such that either (i) $t$ is
labelled with $-1$ or (ii) $t$ is at the end a brick $b_j$ and the
brick $b_{j+1}$ immediately following $b_j$ has the property that
$\sg$ is strictly increasing in all the cells corresponding to
$b_j$ and $b_{j+1}$ and $w$ is constant in all the cells
corresponding to $b_j$ and $b_{j+1}$.  In case (i), $I_w(C)
=(T',\sg',w',L')$ where $T'$ is the result of  replacing the brick
$b$ in $T$ containing $t$ by two bricks $b^*$ and $b^{**}$ where
$b^*$ contains the cell $t$ plus all the cells in $b$ to the left
of $t$ and $b^{**}$ contains all the cells of $b$ to the right of
$t$, $\sg =\sg'$, $w = w'$, and $L'$ is the labelling that results
from $L$ by changing the label of cell $t$ from $-1$ to $1$. In
case (ii), $I_w(C) =(T',\sg',r',L')$ where $T'$ is the result of
replacing the bricks $b_j$ and $b_{j+1}$ in $T$ by a single brick
$b,$ $\sg =\sg'$, $w = w'$, and $L'$ is the labelling that results
from $L$ by changing the label of cell $t$ from $1$ to $-1$. If
neither case (i) or case (ii) applies, then we let $I_w(C) =C$.
For example, if $C$ is the element of $\mathcal{G}_{12}$ pictured
in Figure \ref{figure:wris1}, then $I_w(C)$ is pictured in Figure
\ref{figure:wris2}.

\fig{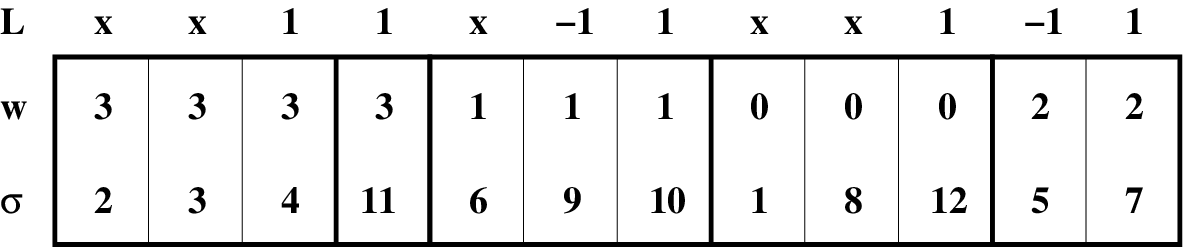}{$I_w(C)$ for $C$ in Figure \ref{figure:wris1}.}

It is easy to see that $I_w$ is a weight-preserving sign-reversing
involution and hence $I_w$ shows that
\begin{equation}\label{wris5}
[n]_{p,q}!\Gamma_w(h_n) = \sum_{C \in \mathcal{G}_{n},I_w(C) = C}
sgn(C) w(C).
\end{equation}

Thus we must examine the fixed points $C = (T,\sg,w,L)$ of $I_w$.
First there can be no $-1$ labels in $L$ so that $sgn(C) =1$.
Moreover,  if $b_j$ and $b_{j+1}$ are two consecutive bricks in $T$
and $t$ is that last cell of $b_j$, then it can not be the case that
$\sg_{t} < \sg_{t+1}$ and $w_t = w_{t+1}$ since otherwise we could
combine $b_j$ and $b_{j+1}$. For any such fixed point, we associate
an element $(\sg,w) \in C_k \wr S_{n}$. For example, a fixed point
of $I_w$ is pictured in Figure \ref{figure:wris3} where
\begin{eqnarray*}
\sg &=& 2~3~4~6~9~10~11~1~8~12~5~7 \ \mbox{and} \\
w &=& 3~3~3~3~1~1~1~2~2~2~3~3.
\end{eqnarray*}
It follows that if cell $t$ is at the end of a brick, then $t \not
\in WRis((\sg,\ep))$. However if $v$ is a cell which is not at the
end of a brick, then our definitions force $\sg_{v} < \sg_{v+1}$
and $w_v =w_{v+1}$  so that $v \in WRis((\sg,\ep))$. Since each
such cell $v$ must be labelled with an $x$, it follows that
$sgn(C)w(C) =
q^{inv(\sg)}p^{coinv(\sg)}r^{||w||}x^{\wris{(\sg,\ep)}}$.
  Vice versa, if
$(\sg,w) \in C_k \wr S_{n}$, then we can create a fixed point $C
=(T,\sg,w,L)$ by having the bricks in $T$ end at cells of the form
$t$ where $t \not \in WRis((\sg,\ep))$, and labelling each cell $t
\in WRis((\sg,\ep))$ with $x$ and labelling the remaining cells
with $1$. Thus we have shown that
\begin{equation*}
[n]_{p,q}!\Gamma_w(h_n) = \sum_{(\sg,w) \in C_k \wr S_{n}}
q^{\inv{\sg}}p^{\coinv{\sg}} r^{||w||} x^{\wris{(\sg,w)}}
\end{equation*}
as desired.

\fig{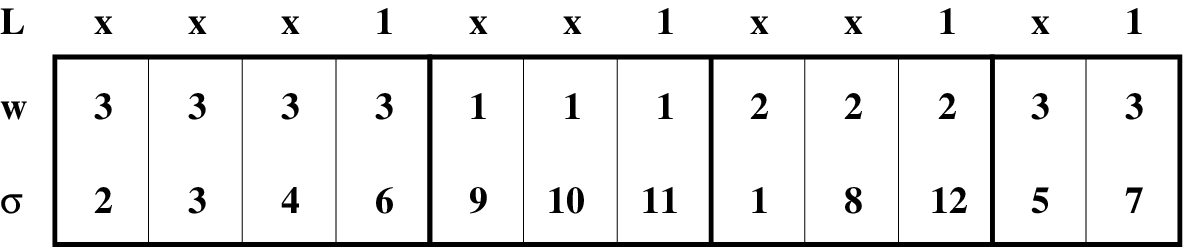}{A fixed point of $I_w$.}

Applying $\Gamma_w$ to the identity $H(t) = (E(-t))^{-1}$, we get
\begin{eqnarray*}
\sum_{n \geq 0} \Gamma_w(h_n) t^n &=& \sum_{n\geq 0} \frac{t^{n}}{[n]_{p,q}!} \sum_{(\sg,w)\in C_k\wr S_{n}} q^{\inv{\sg}}p^{\coinv{\sg}}
r^{||w||}x^{\wris{(\sg,w)}} \\
&=& \frac{1}{1+\sum_{n\geq 1} (-t)^n\Gamma_w(e_n)} \\
&=& \frac{1}{1+\sum_{m\geq 1}(-1)^{m} t^{m}
\frac{(-1)^{m-1}(x-1)^{m-1}p^{\binom{m}{2}}}{[m]_{p,q}!}[k]_{r^m}} \\
&=&
\frac{1-x}{1-x + \sum_{m\geq 1}\frac{p^{\binom{m}{2}}(x-1)^mt^{m}}{[m]_{p,q}!}
[k]_{r^m}}
\end{eqnarray*}
which proves (\ref{eq:pqwris}).
\end{proof}

Next we prove a result which specializes to (\ref{sris}).

\begin{theorem}\label{thm:pqsris} Let $\Upsilon_{{\bf s}} =\{(1~2,0~1)\}$.
For all $k \geq 2$,
\begin{eqnarray}\label{eq:pqsris}
D_k^{\Upsilon_{{\bf s}}}(x,p,q,r,t) &=& \sum_{n \geq 0} \frac{t^n}{[n]_{p,q}!} \sum_{(\sg,w) \in C_k \wr S_n}
q^{\inv{\sg}}p^{\coinv{\sg}}r^{||w||}x^{\sris{(\sg,w)}}
\nonumber \\
&=&
\frac{1-x}{1-x + \sum_{n \geq 1} \frac{p^{\binom{n}{2}}((x-1)t)^n}{[n]_{p,q}!}
r^{\binom{n}{2}}\rbinom{k}{n}}.
\end{eqnarray}
\end{theorem}
 \begin{proof}
Define a ring homomorphism
$\Gamma_s:\Lambda \rightarrow \mathbb{Q}(p,q,r,x)$
by setting $\Gamma_s(e_0) =1$ and
\begin{equation}
\Gamma_s(e_n) = (-1)^{n-1}(x-1)^{n-1}\frac{r^{\binom{n}{2}}\rbinom{k}{n}}{[n]_{p,q}!}
p^{\binom{n}{2}}
\end{equation}
for $n \geq 1$. Then we claim that
\begin{equation}\label{sris1}
[n]_{p,q}!\Gamma_s(h_n) = \sum_{(\sg,w) \in
C_k \wr S_{n}}
q^{\inv{\sg}}p^{\coinv{\sg}} r^{||w||} x^{\sris{(\sg,w)}}
\end{equation}
for all $n \geq 1$.
That is,
\begin{eqnarray}\label{sris2}
&&\ [n]_{p,q}!\Gamma_s(h_{n}) = \nonumber \\
&& \ [n]_{p,q}! \sum_{\mu\vdash n} (-1)^{n-\ell(\mu)}B_{\mu, (n)}\Gamma_s(e_{\mu}) = \nonumber \\
&& \ [n]_{p,q}! \sum_{\mu \vdash n} (-1)^{n-\ell(\mu)} \sum_{(b_1,
\ldots, b_{\ell(\mu)}) \in \mathcal{B}_{\mu,n}}
\prod_{j=1}^{\ell(\mu)} (-1)^{b_j-1}(x-1)^{b_j-1}
\frac{r^{\binom{b_j}{2}}\rbinom{k}{b_j}}{[b_j]_{p,q}!} p^{\binom{b_j}{2}}= \nonumber \\
&& \ \sum_{\mu \vdash n} \sum_{(b_1, \ldots, b_{\ell(\mu)}) \in \mathcal{B}_{\mu,n}} p^{\sum_{j=1}^{\ell(\mu)} \binom{b_j}{2}}\pqbinom{n}{b_1,\ldots,b_{\ell(\mu)}} \prod_{j=1}^{\ell(\mu)} (x-1)^{b_j-1}
r^{\binom{b_j}{2}}\rbinom{k}{b_j}.
\end{eqnarray}

Next we want to give a combinatorial interpretation to (\ref{sris2}).
By Lemma \ref{Carlitz} for each brick tabloid
$T= (b_1, \ldots, b_{\ell(\mu)})$, we can interpret
$p^{\sum_{j=1}^{\ell(\mu)} \binom{b_j}{2}}
\pqbinom{n}{b_1,\ldots,b_{\ell(\mu)}}$ as the sum of the weights of all fillings of $T$ with
a permutation $\sg \in S_{n}$ such that $\sg$ is increasing in each brick
and we weight $\sg$ by $q^{\inv{\sg}}p^{\coinv{\sg}}$.  By Lemma
\ref{Carlitz3},
$$\sum_{0 \leq j_1 \leq \cdots \leq j_n \leq k-n} q^{j_1 + j_2+ \cdots + j_n} =
\qbinom{k}{n}.$$
If we replace each $j_s$ in the sum above by $i_s = j_s+s-1$, then
we see that
\begin{equation}\label{strict}
\sum_{0 \leq i_1 < \cdots < i_n \leq k-1} q^{i_1 + i_2+ \cdots + i_n} =
q^{\binom{n}{2}}\qbinom{k}{n}.
\end{equation}
It follows from (\ref{strict}) that we can interpret the term
$\prod_{j=1}^{\ell(\mu)} r^{\binom{b_j}{2}}\rbinom{k}{b_j}$ as the
sum of the weights of  fillings  $w= w_1 \ldots w_{n}$ where the
elements of $w$ are between 0 and $k-1$ and are strictly
increasing in each brick and where we weight $w$ by $r^{w_1+
\cdots +w_{n}}$. Finally, we interpret $\prod_{j=1}^{\ell(\mu)}
(x-1)^{b_j-1}$ as all ways of picking a label $x$ or $-1$ for each
of the cells of each brick except the final cell. For
completeness, we label the final cell of each brick with $1$. We
shall call all such objects created in this way filled labelled
brick tabloids and let $\mathcal{H}_{n}$ denote the set of all
filled labelled brick tabloids that arise in this way. Thus a $C
\in \mathcal{H}_{n}$ consists of a brick tabloid $T$, a
permutation $\sg \in S_{n}$, a sequence $w \in
\{0,\ldots,k-1\}^{n}$, and a labelling $L$ of the cells of $T$
with elements from $\{x,1,-1\}$ such that
\begin{enumerate}
\item $\sg$ is strictly increasing in each brick, \item $w$ is
strictly increasing  in each brick, \item the final cell of each
brick is labelled with 1, and \item each cell which is not a final
cell of a brick is labelled with x or $-1$.
\end{enumerate}
We then define the weight $w(C)$ of $C$ to be
$q^{\inv{\sg}}p^{\coinv{\sg}} r^{||w||}$ times the product of all
the $x$ labels in $L$ and the sign $sgn(C)$ of $C$ to be
the product of all the $-1$ labels in $L$. For example,
if $n =12$, $k=5$, and $T =(4,3,3,2)$, then Figure \ref{figure:sris1}
pictures such a composite object $C \in \mathcal{H}_{12}$ where
$w(C) = q^{24}p^{42}r^{20}x^5$ and $sgn(C) =-1$.

Thus
\begin{equation}\label{sris4}
[n]_{p,q}!\Gamma_s(h_{n}) = \sum_{C \in \mathcal{H}_{n}}
sgn(C) w(C).
\end{equation}

\fig{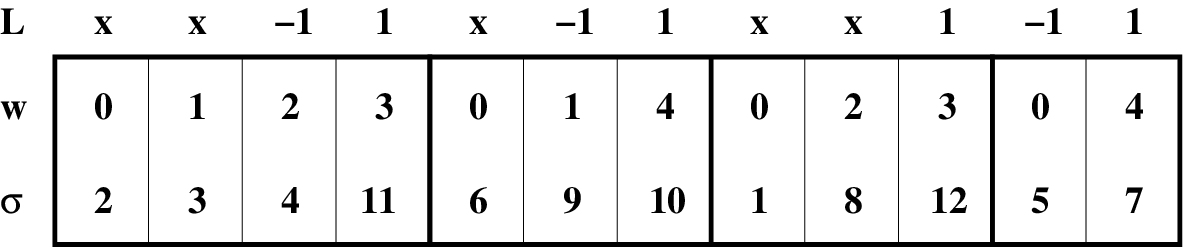}{A composite object $C \in \mathcal{H}_{12}$.}

Next we define a weight-preserving sign-reversing involution
$I_s:\mathcal{H}_{n} \rightarrow \mathcal{H}_{n}$.  To define
$I_s(C)$, we scan the cells of $C =(T,\sg,w,L)$ from left to right
looking for the leftmost cell $t$ such that either (i) $t$ is
labelled with $-1$ or (ii) $t$ is at the end a brick $b_j$ and the
brick $b_{j+1}$ immediately following $b_j$ has the property that
$\sg$ is strictly increasing in all the cells corresponding to
$b_j$ and $b_{j+1}$ and $w$ is strictly in all the cells
corresponding to $b_j$ and $b_{j+1}$.  In case (i), $I_s(C)
=(T',\sg',w',L')$ where $T'$ is the result of  replacing the brick
$b$ in $T$ containing $t$ by two bricks $b^*$ and $b^{**}$ where
$b^*$ contains the cell $t$ plus all the cells in $b$ to the left
of $t$ and $b^{**}$ contains all the cells of $b$ to the right of
$t$, $\sg =\sg'$, $w = w'$, and $L'$ is the labelling that results
from $L$ by changing the label of cell $t$ from $-1$ to $1$. In
case (ii), $I_s(C) =(T',\sg',r',L')$ where $T'$ is the result of
replacing the bricks $b_j$ and $b_{j+1}$ in $T$ by a single brick
$b,$ $\sg =\sg'$, $w = w'$, and $L'$ is the labelling that results
from $L$ by changing the label of cell $t$ from $1$ to $-1$. If
neither case (i) or case (ii) applies, then we let $I_s(C) =C$.
For example, if $C$ is the element of $\mathcal{H}_{12}$ pictured
in Figure \ref{figure:sris1}, then $I_s(C)$ is pictured in Figure
\ref{figure:sris2}.

\fig{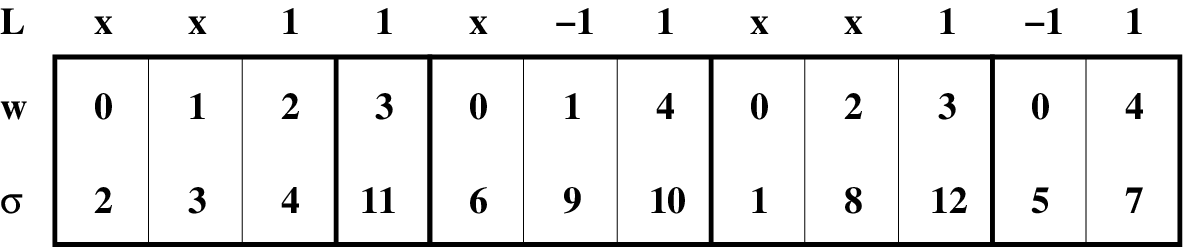}{$I_s(C)$ for $C$ in Figure \ref{figure:sris1}.}

It is easy to see that $I_s$ is a weight-preserving sign-reversing
involution and hence $I_s$ shows that
\begin{equation}\label{sris5}
[n]_{p,q}!\Gamma_s(h_n) = \sum_{C \in \mathcal{H}_{n},I_s(C) = C}
sgn(C) w(C).
\end{equation}

Thus we must examine the fixed points $C = (T,\sg,w,L)$ of $I_s$.
First there can be no $-1$ labels in $L$ so that $sgn(C) =1$.
Moreover,  if $b_j$ and $b_{j+1}$ are two consecutive bricks in $T$
and $t$ is that last cell of $b_j$, then it can not be the case that
$\sg_{t} < \sg_{t+1}$ and $w_t < w_{t+1}$ since otherwise we could
combine $b_j$ and $b_{j+1}$. For any such fixed point, we associate
an element $(\sg,w) \in C_k \wr S_{n}$. For example, a fixed point
of $I_s$ is pictured in Figure \ref{figure:sris3} where
\begin{eqnarray*}
\sg &=& 2~3~4~6~9~10~11~1~8~12~5~7 \ \mbox{and} \\
w &=& 0~1~2~3~0~1~4~0~1~3~3~4.
\end{eqnarray*}
It follows that if cell $t$ is at the end of a brick, then $t \not
\in SRis((\sg,\ep))$. However if $v$ is a cell which is not at the
end of a brick, then our definitions force $\sg_{v} < \sg_{v+1}$
and $w_v <w_{v+1}$  so that $v \in SRis((\sg,\ep))$. Since each
such cell $v$ must be labelled with an $x$, it follows that
$sgn(C)w(C) =
q^{inv(\sg)}p^{coinv(\sg)}r^{||w||}x^{\sris{(\sg,\ep)}}$.
  Vice versa, if
$(\sg,w) \in C_k \wr S_{n}$, then we can create a fixed point $C
=(T,\sg,w,L)$ by having the bricks in $T$ end at cells of the form
$t$ where $t \not \in SRis((\sg,\ep))$, and labelling each cell $t
\in SRis((\sg,\ep))$ with $x$ and labelling the remaining cells
with $1$. Thus we have shown that
\begin{equation*}
[n]_{p,q}!\Gamma_s(h_n) = \sum_{(\sg,w) \in C_k \wr S_{n}}
q^{\inv{\sg}}p^{\coinv{\sg}} r^{||w||} x^{\sris{(\sg,w)}}
\end{equation*}
as desired.

\fig{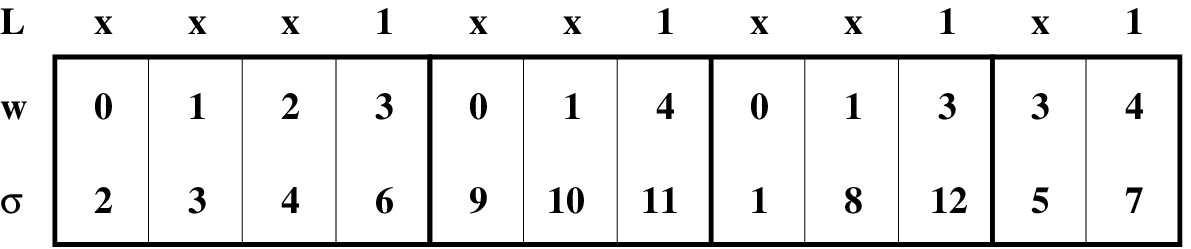}{A fixed point of $I_s$.}

Applying $\Gamma_s$ to the identity $H(t) = (E(-t))^{-1}$, we get
\begin{eqnarray*}
\sum_{n \geq 0} \Gamma_s(h_n) t^n &=& \sum_{n\geq 0} \frac{t^{n}}{[n]_{p,q}!} \sum_{(\sg,w)\in C_k\wr S_{n}} q^{\inv{\sg}}p^{\coinv{\sg}}
r^{||w||}x^{\sris{(\sg,w)}} \\
&=& \frac{1}{1+\sum_{n\geq 1} (-t)^n\Gamma_s(e_n)} \\
&=& \frac{1}{1+\sum_{m\geq 1}(-1)^{m} t^{m}
\frac{(-1)^{m-1}(x-1)^{m-1}p^{\binom{m}{2}}}{[m]_{p,q}!}r^{\binom{m}{2}}\rbinom{k}{m}} \\
&=&
\frac{1-x}{1-x + \sum_{m\geq 1}\frac{p^{\binom{m}{2}}(x-1)^mt^{m}}{[m]_{p,q}!}
r^{\binom{m}{2}}\rbinom{k}{m}}
\end{eqnarray*}
which proves (\ref{eq:pqsris}).
\end{proof}

We end this section by proving a generating function
which specializes to (\ref{Up}).

\begin{theorem}\label{thm:pqUris}
Let
$\Upsilon_{{\bf d}} = \{(1~2,0~1), (1~2,1~0)\}$
For all $k \geq 2$,
\begin{eqnarray}\label{eq:pqUris}
D_k^{\Upsilon_{{\bf d}}}(x,p,q,r,t) &=& \sum_{n \geq 0} \frac{t^n}{[n]_{p,q}!} \sum_{(\sg,w) \in C_k \wr S_n}
q^{\inv{\sg}}p^{\coinv{\sg}}x^{\Umch{(\sg,w)}}
\nonumber \\
&=&
\frac{(k-1)(1-x)}{(k-1)(1-x) + k\sum_{n \geq 1} \frac{p^{\binom{n}{2}}((k-1)(x-1)t)^n}{[n]_{p,q}!}}.
\end{eqnarray}
\end{theorem}
 \begin{proof}
Define a ring homomorphism
$\Gamma_U:\Lambda \rightarrow \mathbb{Q}(p,q,r,x)$ by setting $\Gamma_U(e_0) =1$ and
\begin{equation}
\Gamma_U(e_n) = (-1)^{n-1}(x-1)^{n-1}\frac{k(k-1)^{n-1}}{[n]_{p,q}!}
p^{\binom{n}{2}}
\end{equation}
for $n \geq 1$. Then we claim that
\begin{equation}\label{Uris1}
[n]_{p,q}!\Gamma_U(h_n) = \sum_{(\sg,w) \in
C_k \wr S_{n}}
q^{\inv{\sg}}p^{\coinv{\sg}} x^{\Umch{(\sg,w)}}
\end{equation}
for all $n \geq 1$.
That is,
\begin{eqnarray}\label{Uris2}
&&\ [n]_{p,q}!\Gamma_U(h_{n}) = \nonumber \\
&& \ [n]_{p,q}! \sum_{\mu\vdash n} (-1)^{n-\ell(\mu)}B_{\mu, (n)}\Gamma_U(e_{\mu}) = \nonumber \\
&& \ [n]_{p,q}! \sum_{\mu \vdash n} (-1)^{n-\ell(\mu)} \sum_{(b_1,
\ldots, b_{\ell(\mu)}) \in \mathcal{B}_{\mu,n}}
\prod_{j=1}^{\ell(\mu)} (-1)^{b_j-1}(x-1)^{b_j-1}
\frac{k(k-1)^{b_j-1}}{[b_j]_{p,q}!} p^{\binom{b_j}{2}}= \nonumber \\
&& \ \sum_{\mu \vdash n} \sum_{(b_1, \ldots, b_{\ell(\mu)}) \in \mathcal{B}_{\mu,n}} p^{\sum_{j=1}^{\ell(\mu)} \binom{b_j}{2}}\pqbinom{n}{b_1,\ldots,b_{\ell(\mu)}} \prod_{j=1}^{\ell(\mu)} (x-1)^{b_j-1}
k(k-1)^{b_j-1}.
\end{eqnarray}

Next we want to give a combinatorial interpretation to
(\ref{Uris2}). By Lemma \ref{Carlitz} for each brick tabloid $T=
(b_1, \ldots, b_{\ell(\mu)})$, we can interpret
$p^{\sum_{j=1}^{\ell(\mu)} \binom{b_j}{2}}
\pqbinom{n}{b_1,\ldots,b_{\ell(\mu)}}$ as the sum of the weights
of all fillings of $T$ with a permutation $\sg \in S_{n}$ such
that $\sg$ is increasing in each brick and we weight $\sg$ by
$q^{\inv{\sg}}p^{\coinv{\sg}}$. For any $n$, there are
$k(k-1)^{n-1}$ words $w=w_1 w_2 \ldots w_n \in [k]^n$ such that
for $1 \leq i <n$, $w_i \neq w_{i+1}$. That is, we have $k$
choices for the first letter $w_1$, but then, for any given $i$,
we have only $k-1$ choices for $w_{i+1}$ since $w_{i+1}$ cannot
equal $w_i$. Thus we can interpret $\prod_{j=1}^{\ell(\mu)}
k(k-1)^{b_j-1}$ as the number of words $w_1 \ldots w_n$ so that
within any brick, there are never two consecutive letters of $w$
which are equal. Finally, we interpret $\prod_{j=1}^{\ell(\mu)}
(x-1)^{b_j-1}$ as all ways of picking a label of the cells of each
brick except the final cell with either an $x$ or a $-1$. For
completeness, we label the final cell of each brick with $1$. We
shall call all such objects created in this way filled labelled
brick tabloids and let $\mathcal{K}_{n}$ denote the set of all
filled labelled brick tabloids that arise in this way.  Thus a $C
\in \mathcal{K}_{n}$ consists of a brick tabloid $T$, a
permutation $\sg \in S_{n}$, a sequence $w \in
\{0,\ldots,k-1\}^{n}$, and a labelling $L$ of the cells of $T$
with elements from $\{x,1,-1\}$ such that
\begin{enumerate}
\item $\sg$ is strictly increasing in each brick, \item $w$ is
such that there are never two consecutive letters that lie in the
same brick which are equal, \item the final cell of each brick is
labelled with 1, and \item each cell which is not a final cell of
a brick is labelled with x or $-1$.
\end{enumerate}
We then define the weight $w(C)$ of $C$ to be
$q^{\inv{\sg}}p^{\coinv{\sg}}$ times the product of all the $x$
labels in $L$ and the sign $sgn(C)$ of $C$ to be the product of all
the $-1$ labels in $L$. For example, if $n =12$, $k=5$, and $T
=(4,3,3,2)$, then Figure \ref{figure:Uris1} pictures such a
composite object $C \in \mathcal{K}_{12}$ where $w(C) =
q^{24}p^{42}x^5$ and $sgn(C) =-1$.

Thus
\begin{equation}\label{Uris4}
[n]_{p,q}!\Gamma_U(h_{n}) = \sum_{C \in \mathcal{K}_{n}}
sgn(C) w(C).
\end{equation}

\fig{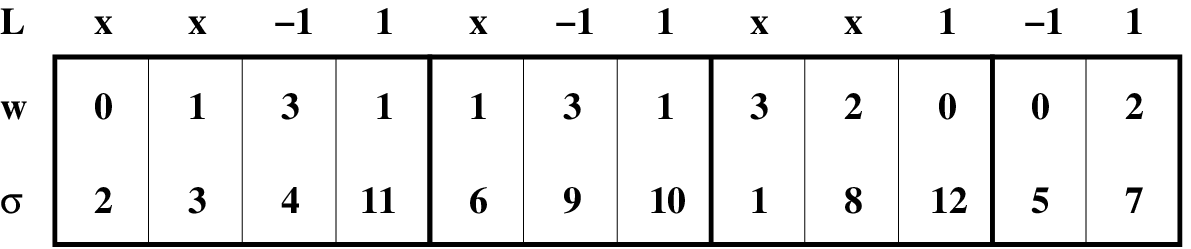}{A composite object $C \in \mathcal{K}_{12}$.}

Next we define a weight-preserving sign-reversing involution
$I_U:\mathcal{K}_{n} \rightarrow \mathcal{K}_{n}$.  To define
$I_U(C)$, we scan the cells of $C =(T,\sg,w,L)$ from left to right
looking for the leftmost cell $t$ such that either (i) $t$ is
labelled with $-1$ or (ii) $t$ is at the end a brick $b_j$ and the
brick $b_{j+1}$ immediately following $b_j$ has the property that
 $\sg$ is strictly increasing in all the cells corresponding to
$b_j$ and $b_{j+1}$ and there are never two consecutive elements
of $w$ that are equal in all the cells corresponding to $b_j$ and
$b_{j+1}$.  In case (i), $I_U(C) =(T',\sg',w',L')$ where $T'$ is
the result of  replacing the brick $b$ in $T$ containing $t$ by
two bricks $b^*$ and $b^{**}$ where $b^*$ contains the cell $t$
plus all the cells in $b$ to the left of $t$ and $b^{**}$ contains
all the cells of $b$ to the right of $t$, $\sg =\sg'$, $w = w'$,
and $L'$ is the labelling that results from $L$ by changing the
label of cell $t$ from $-1$ to $1$. In case (ii), $I_U(C)
=(T',\sg',r',L')$ where $T'$ is the result of replacing the bricks
$b_j$ and $b_{j+1}$ in $T$ by a single brick $b,$ $\sg =\sg'$, $w
= w'$, and $L'$ is the labelling that results from $L$ by changing
the label of cell $t$ from $1$ to $-1$. If neither case (i) or
case (ii) applies, then we let $I_U(C) =C$. For example, if $C$ is
the element of $\mathcal{K}_{12}$ pictured in Figure
\ref{figure:Uris1}, then $I_U(C)$ is pictured in Figure
\ref{figure:Uris2}.

\fig{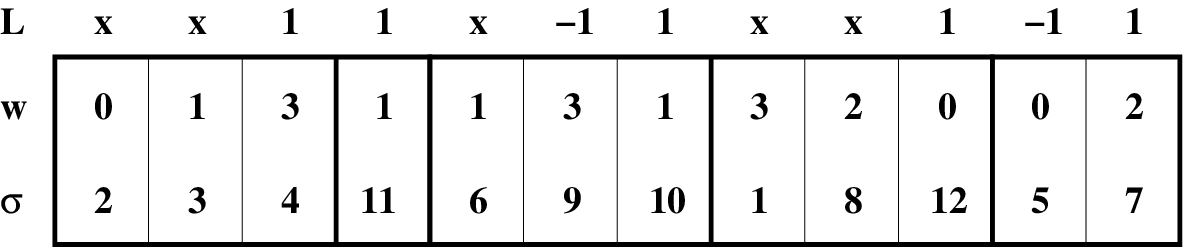}{$I_U(C)$ for $C$ in Figure \ref{figure:Uris1}.}

It is easy to see that $I_U$ is a weight-preserving sign-reversing
involution and hence $I_U$ shows that
\begin{equation}\label{Uris5}
[n]_{p,q}!\Gamma_U(h_n) = \sum_{C \in \mathcal{K}_{n},I_U(C) = C}
sgn(C) w(C).
\end{equation}

Thus we must examine the fixed points $C = (T,\sg,w,L)$ of $I_U$.
First there can be no $-1$ labels in $L$ so that $sgn(C) =1$.
Moreover,  if $b_j$ and $b_{j+1}$ are two consecutive bricks in $T$
and $t$ is the last cell of $b_j$, then it can not be the case that
$\sg_{t} < \sg_{t+1}$ and $w_t \neq w_{t+1}$ since otherwise we
could combine $b_j$ and $b_{j+1}$. For any such fixed point, we
associate an element $(\sg,w) \in C_k \wr S_{n}$. For example, a
fixed point of $I_U$ is pictured in Figure \ref{figure:Uris3} where
\begin{eqnarray*}
\sg &=& 2~3~4~6~9~10~11~1~8~12~5~7 \ \mbox{and} \\
w &=& 0~1~3~1~1~0~3~3~2~3~3~0.
\end{eqnarray*}
It follows that if cell $t$ is at the end of a brick, then there
is no $\Upsilon$-match in $(\sg,w)$ starting at position $t$.
However if $v$ is a cell which is not at the end of a brick, then
our definitions force $\sg_{v} < \sg_{v+1}$ and $w_v \neq w_{v+1}$
so that there is $\Upsilon$-match in $(\sg,w)$ starting at
position $v$. Since each such cell $v$ must be labelled with an
$x$, it follows that $sgn(C)w(C) =
q^{inv(\sg)}p^{coinv(\sg)}x^{\Umch{(\sg,\ep)}}$.
  Vice versa, if
$(\sg,w) \in C_k \wr S_{n}$, then we can create a fixed point $C
=(T,\sg,w,L)$ by having the bricks in $T$ end at cells of the form
$t$ where there is no $\Upsilon$-match in $(\sg,w)$ starting at
position $t$, and labelling each cell $t$ where there is an
$\Upsilon$-match in $(\sg,w)$ starting at position $t$ with $x$
and labelling the remaining cells with $1$. Thus we have shown
that
\begin{equation*}
[n]_{p,q}!\Gamma_U(h_n) = \sum_{(\sg,w) \in C_k \wr S_{n}}
q^{\inv{\sg}}p^{\coinv{\sg}} x^{\Umch{(\sg,w)}}
\end{equation*}
as desired.

\fig{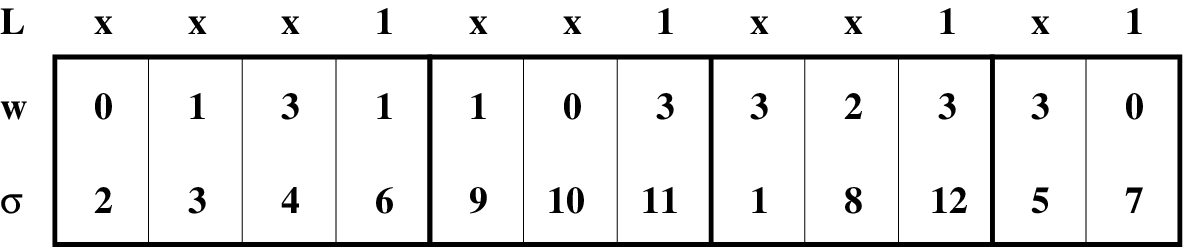}{A fixed point of $I_U$.}

Applying $\Gamma_U$ to the identity $H(t) = (E(-t))^{-1}$, we get
\begin{eqnarray*}
\sum_{n \geq 0} \Gamma_U(h_n) t^n &=& \sum_{n\geq 0} \frac{t^{n}}{[n]_{p,q}!} \sum_{(\sg,w)\in C_k\wr S_{n}} q^{\inv{\sg}}p^{\coinv{\sg}}
x^{\Umch{(\sg,w)}} \\
&=& \frac{1}{1+\sum_{n\geq 1} (-t)^n\Gamma_U(e_n)} \\
&=& \frac{1}{1+\sum_{m\geq 1}(-1)^{m} t^{m}
\frac{(-1)^{m-1}(x-1)^{m-1}p^{\binom{m}{2}}}{[m]_{p,q}!}k(k-1)^{m-1} }\\
&=&
\frac{(k-1)(1-x)}{(k-1)(1-x) + k\sum_{m\geq 1}\frac{p^{\binom{m}{2}}
((k-1)(x-1)t)^m}{[m]_{p,q}!}}
\end{eqnarray*}
which proves (\ref{eq:pqUris}).
\end{proof}

\section{Distribution of non-overlapping $\Upsilon$-bi-matches}

In this section we provide arguments similar to those
in~\cite[Sect. 4]{K} to determine the generating function for
$\Ulap{(\sg,w)}$, the maximum number of non-overlapping
$\Upsilon$-bi-matches. That is, suppose that $\Upsilon \subseteq
C_k \wr S_j$. Recall that
\begin{equation}\label{defNU1}
N_k^{\Upsilon}(x,p,q,r,t) = \sum_{n \geq 0} \frac{t^n}{[n]_{p,q}!}
\sum_{(\sg,w) \in C_k \wr S_n} q^{\inv{\sg}}p^{\coinv{\sg}}r^{||w||}
x^{\Ulap{(\sg,w)}},
\end{equation}
and
\begin{equation}\label{AU1}
A_k^{\Upsilon}(p,q,r,t) = \sum_{n \geq 0} \frac{t^n}{[n]_{p,q}!}
\sum_{(\sg,w) \in C_k \wr S_n} q^{\inv{\sg}}p^{\coinv{\sg}}r^{||w||}
\chi( \Umch{(\sg,w)}=0).
\end{equation}
Let $(C_k \wr S_n)_{\Umch{end}}$ denote the set of all
$(\sg,w)$ such that $(\sg,w)$ has exactly one
$\Upsilon$-match which occurs at the end of $(\sg,w)$, i.e. the unique $\Upsilon$-match in $(\sg,w)$ starts
at position $n-j+1$. We then let
\begin{equation}\label{BU1}
B_k^{\Upsilon}(p,q,r,t) = \sum_{n \geq 1} \frac{t^n}{[n]_{p,q}!}
\sum_{(\sg,w) \in (C_k \wr S_n)_{\Umch{end}}}
q^{\inv{\sg}}p^{\coinv{\sg}}r^{||w||}.
\end{equation}

\begin{lemma}\label{lem01} For all $k \ge2$, we  have $B_k^{\Upsilon}(p,q,r,t)=([k]_rt-1)
A_k^{\Upsilon}(p,q,r,t)+1$.
\end{lemma}
\begin{proof}
Suppose that $(\sg,w) \in C_k \wr S_{n-1}$, let $\sg^j$ be the
result of replacing $j, \ldots ,n-1$ in $\sg$ by $j+1,\ldots, n$
respectively and then adding $j$ at the end. For example, if $\sg  =
1~3~4~2$, then $\sg^2 = 1~4~5~3~2$. Clearly,
\begin{eqnarray*}
&&\sum_{i=0}^{k-1} \sum_{j=1}^{n} q^{\inv{(\sg^j,wi)}}
p^{\coinv{(\sg^j,wi)}} r^{||wi||} = \\
&& (1+r+\cdots + r^{k-1})(p^{n-1} +qp^{n-2} +
\cdots + pq^{n-2}+q^{n-1}) q^{\inv{(\sg,w)}}  p^{\coinv{(\sg,w)}} r^{||w||}
= \\
&&[k]_r [n]_{p,q} q^{\inv{(\sg,w)}}  p^{\coinv{(\sg,w)}} r^{||w||}.
\end{eqnarray*}

Now if $(\sg,w) \in C_k \wr S_{n-1}$ and
$\Umch{(\sg,w)} =0$, then for any $0 \leq i \leq k-1$ and $1 \leq j \leq n-1$,
the pair $(\sg^j,wi)$ either has no $\Upsilon$-match or
has exactly one $\Upsilon$-match which occurs at the end.
It follows that
\begin{equation}\label{A=A+B}
[k]_r [n]_{p,q}! A_k^{\Upsilon}(p,q,r,t)|_{\frac{t^{n-1}}{[n-1]_{p,q}!}} =
 A_k^{\Upsilon}(p,q,r,t)|_{\frac{t^{n}}{[n]_{p,q}!}}+
 B_k^{\Upsilon}(p,q,r,t)|_{\frac{t^{n}}{[n]_{p,q}!}}.
\end{equation}
If we multiply both sides of (\ref{A=A+B}) by $\frac{t^n}{[n]_{p,q}!}$ and
sum for $n \geq 1$, we get that
\begin{equation*}
[k]_rt A_k^{\Upsilon}(p,q,r,t) = A_k^{\Upsilon}(p,q,r,t)-1 + B_k^{\Upsilon}(p,q,r,t)
\end{equation*}
or that
\begin{equation*}
B_k^{\Upsilon}(p,q,r,t) = 1+ ([k]_rt -1) A_k^{\Upsilon}(p,q,r,t).
\end{equation*}
\end{proof}

\begin{theorem}\label{overlap} For all $\Upsilon\  \subseteq
C_k \wr S_j$ and $k \geq 2$,
\begin{equation}\label{eq:nlap}
N_k^{\Upsilon}(x,p,q,r,t) = \frac{A_k^{\Upsilon}(p,q,r,t)}{1
-x(1+([k]_rt -1)A_k^{\Upsilon}(p,q,r,t))}.
\end{equation}
\end{theorem}

\begin{proof}

Suppose that $\Ulap{(\sg,w)}=i\geq 0$. One can read any such
$(\sg,w)$ from left to right making a cut right after a
$\Upsilon$-bi-occurrence counted by $\Ulap{(\sg,w)}$. As the result,
one obtains $i$ signed words which have exactly one $\Upsilon$-match and
that $\Upsilon$-match occurs at the end of the word  that is
followed by a possibly empty word that has no $\Upsilon$-matches. In terms
of generating functions, this says that
\begin{eqnarray*}
&&N_k^{\Upsilon}(x,p,q,r,t) = \\
&& A_k^{\Upsilon}(p,q,r,t) +
xB_k^{\Upsilon}(p,q,r,t) A_k^{\Upsilon}(p,q,r,t)+
(x B_k^{\Upsilon}(p,q,r,t))^2 A_k^{\Upsilon}(p,q,r,t)+\cdots=\\
&&\frac{A_k^{\Upsilon}(p,q,r,t)}{1-x B_k^{\Upsilon}(p,q,r,t)}.
\end{eqnarray*}
The result then follows from Lemma~\ref{lem01}.
\end{proof}

Using our results in Section 3, we
immediately have the following corollaries setting $x =0$ in
our formulas for $D_k^\Upsilon(x,p,q,r,t)$.

\begin{corollary}\label{thm:0ris}
Let $\Upsilon_{{\bf r}} = \{(1~2,0~0),(1~2,0~1)\}$. Then for all $k \geq 2$,
\begin{equation}\label{eq:0ris}
A_k^{\Upsilon_{{\bf r}}}(p,q,r,t) = \frac{1}{1+ \sum_{n \geq 1}
\frac{p^{\binom{n}{2}} (-t)^n}{[n]_{p,q}!}\rbinom{n+k-1}{n}}.
\end{equation}
\end{corollary}

\begin{corollary}\label{thm:0wris}
Let $\Upsilon_{{\bf w}} = \{(1~2,0~0)\}$. Then for all $k \geq 2$,
\begin{equation}\label{eq:0wris}
A_k^{\Upsilon_{{\bf w}}}(p,q,r,t) = \frac{1}{1+ \sum_{n \geq 1}
\frac{p^{\binom{n}{2}} (-t)^n}{[n]_{p,q}!}[k]_{r^n}}.
\end{equation}
\end{corollary}

\begin{corollary}\label{thm:0sris}
Let $\Upsilon_{{\bf s}} = \{(1~2,0~1)\}$. Then for all $k \geq 2$,
\begin{equation}\label{eq:0sris}
A_k^{\Upsilon_{{\bf s}}}(p,q,r,t) = \frac{1}{1+ \sum_{n \geq 1}
\frac{p^{\binom{n}{2}}
(-t)^n}{[n]_{p,q}!}r^{\binom{n}{2}}\rbinom{n}{k}}.
\end{equation}
\end{corollary}

Thus it follows that we can obtain the generating functions
for $N_k^{\Upsilon_{{\bf a}}}(x,p,q,r,t)$ for ${\bf a}\in
\{{\bf r},{\bf s},{\bf w}\}$ immediately from
Theorem \ref{overlap}.

Similarly, it follows from Theorem 8 that we have the following corollary.

\begin{corollary}\label{thm:0dris}
Let $\Upsilon_{{\bf d}} = \{(1~2,0~1),(1~2,1~0)\}$. Then for all $k \geq 2$,
\begin{equation}\label{eq:0dris}
A_k^{\Upsilon_{{\bf d}}}(p,q,1,t) = \frac{k-1}{k-1+ k\sum_{n \geq 1}
\frac{p^{\binom{n}{2}} (-(k-1)t)^n}{[n]_{p,q}!}}.
\end{equation}
\end{corollary}

Thus we can obtain $N_k^{\Upsilon_{{\bf d}}}(x,p,q,1,t)$ from
Theorem \ref{overlap}.

\section{More generating functions}

For ${\bf a} \in \{{\bf r},{\bf w},{\bf s},{\bf d}\}$,
 let $onenlap^{\Upsilon_{{\bf a}}}(C_k \wr S_n)$ denote
the set of permutations $(\sg,w) \in C_k \wr S_n$ such that
$\Upsilon_{\bf a}\mbox{-nlap}(\sg,w) =1$, $onemch^{\Upsilon_{{\bf
a}}}(C_k \wr S_n)$ denote the set of permutations $(\sg,w) \in C_k\
\wr S_n$ such that $\Upsilon_{\bf a}\mbox{-mch}(\sg,w) =1$, and
$twomch^{\Upsilon_{{\bf a}}}(C_k \wr S_n)$ denote the set of
permutations $(\sg,w) \in C_k\wr S_n$ such that $\Upsilon_{\bf
a}\mbox{-mch}(\sg,w) =2$. It is easy to see that
$$onemch^{\Upsilon_{{\bf a}}}(C_k \wr S_n) \subseteq
 onenlap^{\Upsilon_{{\bf a}}}(C_k \wr S_n).$$
Now define
\begin{equation}\label{defU}
\mathcal{U}^{\Upsilon_{{\bf a}}}_{n,k} = onenlap^{\Upsilon_{{\bf a}}}(C_k \wr S_n) -
onemch^{\Upsilon_{{\bf a}}}(C_k \wr S_n).
\end{equation}
Thus $\mathcal{U}^{\Upsilon_{{\bf a}}}_{n,k}$ consists of those
permutations $(\sg,w)$ such that there is an $s$ with $1 \leq s <
n-1$ such that $(\sg,w)$ has a $\Upsilon_{{\bf a}}$-match starting
at positions $s$ and $s+1$ and these are the only $\Upsilon_{{\bf
a}}$-matches in $(\sg,w)$. For example, $\mathcal{U}^{\Upsilon_{{\bf
r}}}_{n,k}$ consists of those permutations $(\sg,w)\in C_k \wr S_n$
such that $Ris(\sg,w) = \{s,s+1\}$ for some $s$. Similarly
$\mathcal{U}^{\Upsilon_{{\bf w}}}_{n,k}$ consists of those
permutations $(\sg,w) \in C_k \wr S_n$ such that $WRis(\sg,w) =
\{s,s+1\}$ for some $s$ and $\mathcal{U}^{\Upsilon_{{\bf s}}}_{n,k}$
consists of those permutations $(\sg,w) \in C_k \wr S_n$ such that
$SRis(\sg,w) = \{s,s+1\}$ for some $s$. It is also the case that
$$\mathcal{U}^{\Upsilon_{{\bf a}}}_{n,k} \subseteq twomch^{\Upsilon_{{\bf a}}}(C_k \wr S_n).$$
Now define
\begin{equation}\label{defV}
\mathcal{V}^{\Upsilon_{{\bf a}}}_{n,k} = twomch^{\Upsilon_{{\bf a}}}(C_k \wr S_n) -
\mathcal{U}^{\Upsilon_{{\bf a}}}_{n,k}.
\end{equation}
Then $\mathcal{V}^{\Upsilon_{{\bf a}}}_{n,k}$ consists of those
permutations $(\sg,w) \in C_k \wr S_n$ such that $(\sg,w)$ has
exactly two $\Upsilon_{{\bf a}}$-matches and those $\Upsilon_{{\bf
a}}$-matches do not overlap.  Thus, $\mathcal{V}^{\Upsilon_{{\bf
r}}}_{n,k}$ consists of those permutations $(\sg,w) \in C_k \wr S_n$
such that $WRis(\sg,w) = \{i,j \}$ where $i+2 \leq j$. Similarly,
$\mathcal{V}^{\Upsilon_{{\bf w}}}_{n,k}$ consists of those
permutations $(\sg,w) \in C_k \wr S_n$ such that $Ris(\sg,w) = \{i,j
\}$ where $i+2 \leq j$ and $\mathcal{V}^{\Upsilon_{{\bf s}}}_{n,k}$
consists of those permutations $(\sg,w) \in C_k \wr S_n$ such that
$SRis(\sg,w) = \{i,j \}$ where $i+2 \leq j$.

We define
\begin{eqnarray}
R_k^{\Upsilon_{{\bf a}}}(p,q,r,t) &=& \sum_{n \geq 0}
\frac{t^n}{[n]_{p,q}!} \sum_{(\sg,w) \in \mathcal{U}^{\Upsilon_{{\bf
a}}}_{n,k}} q^{\inv{\sg}}p^{\coinv{\sg}}r^{||w||} \nonumber
\end{eqnarray}
and
\begin{eqnarray}
S_k^{\Upsilon_{{\bf a}}}(p,q,r,t) &=& \sum_{n \geq 0}
\frac{t^n}{[n]_{p,q}!} \sum_{(\sg,w) \in \mathcal{V}^{\Upsilon_{{\bf
a}}}_{n,k}} q^{\inv{\sg}}p^{\coinv{\sg}}r^{||w||}. \nonumber
\end{eqnarray}

Then from our definitions
\begin{eqnarray}
R_k^{\Upsilon_{{\bf a}}}(p,q,r,t) &=& [N_k^{\Upsilon_{{\bf
a}}}(x,p,q,r,t) - D_k^{\Upsilon_{{\bf a}}}(x,p,q,r,t)]|_x \nonumber
\end{eqnarray}
and
\begin{eqnarray}
S_k^{\Upsilon_{{\bf a}}}(p,q,r,t) &=& D_k^{\Upsilon_{{\bf
a}}}(x,p,q,r,t)|_{x^2} - R_k^{\Upsilon_{{\bf a}}}(p,q,r,t).
\nonumber
\end{eqnarray}

We shall show that we can easily find the generating functions
$R_k^{\Upsilon_{{\bf a}}}(p,q,r,t)$ and $S_k^{\Upsilon_{{\bf
a}}}(p,q,r,t)$ for ${\bf a} \in \{{\bf r},{\bf w},{\bf s}\}$ and the
generating function $R_k^{\Upsilon_{{\bf d}}}(p,q,1,t)$ and
$S_k^{\Upsilon_{{\bf d}}}(p,q,1,t)$. That is, consider the case when
${\bf a} ={\bf r}$. Then
\begin{equation}\label{2eq:nlap}
N_k^{\Upsilon_{{\bf r}}}(x,p,q,r,t) =
\frac{A_k^{\Upsilon_{{\bf r}}}(p,q,r,t)}{1 -x(1+([k]_rt -1)A_k^{\Upsilon_{{\bf r}}}(p,q,r,t))}
\end{equation}
so that
\begin{equation}\label{3eq:nlap}
N_k^{\Upsilon_{{\bf r}}}(x,p,q,r,t)|_x = A_k^{\Upsilon_{{\bf
r}}}(p,q,r,t)(1+([k]_rt -1)A_k^{\Upsilon_{{\bf r}}}(p,q,r,t)).
\end{equation}
In our case,
$$A_k^{\Upsilon_{{\bf r}}}(p,q,r,t) = \frac{1}{P_k^{\Upsilon_{{\bf r}}}(t)}$$
where
\begin{equation}\label{4eq:nlap}
P_k^{\Upsilon_{{\bf r}}}(t)= 1+\sum_{n \geq 1} \frac{(-t)^n}{[n]_{p,q}!}\rbinom{n+k-1}{n}.
\end{equation}
Thus
\begin{equation}\label{5eq:nlap}
N_k^{\Upsilon_{{\bf r}}}(x,p,q,r,t)|_x =
\frac{([k]_rt -1) + P_k^{\Upsilon_{{\bf r}}}(t)}{(P_k^{\Upsilon_{{\bf r}}}(t))^2}.
\end{equation}
On the other hand, it follows from our results in section 3 that
\begin{equation*}\label{6eq:nlap}
D_k^{\Upsilon_{{\bf r}}}(x,p,q,r,t) = \frac{1}{1-\sum_{n \geq 1} (x-1)^{n-1}
\frac{t^n}{[n]_{p,q}!}\rbinom{n+k-1}{n}}.
\end{equation*}
Thus
\begin{equation*}\label{7eq:nlap}
D_k^{\Upsilon_{{\bf r}}}(x,p,q,r,t)|_x = \sum_{m \geq 1}
\left( \sum_{n \geq 1} (x-1)^{n-1}
\frac{t^n}{[n]_{p,q}!}\rbinom{n+k-1}{n}\right)^m|_x .
\end{equation*}
However,
\begin{eqnarray*}\label{8eq:nlap}
&&\sum_{n \geq 1} (x-1)^{n-1}
\frac{t^n}{[n]_{p,q}!}\rbinom{n+k-1}{n} = \\
&&F_k^{\Upsilon_{{\bf r}}}(t) +xG_k^{\Upsilon_{{\bf r}}}(t) +
x^2H_k^{\Upsilon_{{\bf r}}}(t) + O(x^3),
\end{eqnarray*}
where
\begin{equation}\label{9eq:nlap}
F_k^{\Upsilon_{{\bf r}}}(t) = \sum_{n \geq 1} (-1)^{n-1}
\frac{t^n}{[n]_{p,q}!}\rbinom{n+k-1}{n},
\end{equation}
\begin{equation}\label{10eq:nlap}
G_k^{\Upsilon_{{\bf r}}}(t) = \sum_{n \geq 2} (-1)^{n-2} (n-1)
\frac{t^n}{[n]_{p,q}!}\rbinom{n+k-1}{n},
\end{equation}
and
\begin{equation}\label{11eq:nlap}
H_k^{\Upsilon_{{\bf r}}}(t) = \sum_{n \geq 3} (-1)^{n-3}
\binom{n-1}{2} \frac{t^n}{[n]_{p,q}!}\rbinom{n+k-1}{n}.
\end{equation}

Thus since
$$(F^{\Upsilon_{{\bf r}}}(t)+xG^{\Upsilon_{{\bf r}}}(t)+O(x^2))^m|_x = mG^{\Upsilon_{{\bf r}}}(t)(F^{\Upsilon_{{\bf r}}}(t))^{m-1},$$
we have
\begin{eqnarray*}\label{11aeq:nlap}
D_k^{\Upsilon_{{\bf r}}}(x,p,q,r,t)|_x &=& G^{\Upsilon_{{\bf r}}}(t) \sum_{m \geq 1}m (F^{\Upsilon_{{\bf r}}}(t))^{m-1} \nonumber \\
&=&
\frac{G^{\Upsilon_{{\bf r}}}(t)}{(1-F^{\Upsilon_{{\bf r}}}(t))^2}.
\end{eqnarray*}
However
$$1- F_k^{\Upsilon_{{\bf r}}}(t) = 1+ \sum_{n \geq 1} (-1)^{n}
\frac{t^n}{[n]_{p,q}!}\rbinom{n+k-1}{n} = P_k^{\Upsilon_{{\bf r}}}(t).$$
Thus
\begin{equation}\label{12eq:nlap}
D_k^{\Upsilon_{{\bf r}}}(x,p,q,r,t)|_x = \frac{G_k^{\Upsilon_{{\bf r}}}(t)}{(P_k^{\Upsilon_{{\bf r}}}(t))^2}.
\end{equation}
It follows that
\begin{equation}\label{13eq:nlap}
R_k^{\Upsilon_{{\bf r}}}(p,q,r,t) = \frac{([k]_rt -1) + P_k^{\Upsilon_{{\bf r}}}(t) -
G_k^{\Upsilon_{{\bf r}}}(t)}{(P_k^{\Upsilon_{{\bf r}}}(t))^2}.
\end{equation}
Similarly,
\begin{eqnarray}\label{14eq:nlap}
D_k^{\Upsilon_{{\bf r}}}(x,p,q,r,t)|_{x^2} &=& \sum_{m \geq 1} m H_k^{\Upsilon_{{\bf r}}}(t)(F_k^{\Upsilon_{{\bf r}}}(t))^{m-1} +
\binom{m}{2} (G_k^{\Upsilon_{{\bf r}}}(t))^2 (F_k^{\Upsilon_{{\bf r}}}(t))^{m-2}\nonumber \\
&=& H_k^{\Upsilon_{{\bf r}}}(t) \sum_{m \geq 1} m (F_k^{\Upsilon_{{\bf r}}}(t))^{m-1} +
(G_k^{\Upsilon_{{\bf r}}}(t))^2 \sum_{m\geq 2}\binom{m}{2}  (F_k^{\Upsilon_{{\bf r}}}(t))^{m-2} \nonumber \\
&=& \frac{H_k^{\Upsilon_{{\bf r}}}(t)}{(1-F_k^{\Upsilon_{{\bf r}}}(t))^2}+ \frac{(G_k^{\Upsilon_{{\bf r}}}(t))^2}{(1-F_k^{\Upsilon_{{\bf r}}}(t))^3} \nonumber \\
&=& \frac{H_k^{\Upsilon_{{\bf r}}}(t)}{((P_k^{\Upsilon_{{\bf r}}}(t))^2}+ \frac{(G_k^{\Upsilon_{{\bf r}}}(t))^2}{(P_k^{\Upsilon_{{\bf r}}}(t))^3} \nonumber \\
&=& \frac{H_k^{\Upsilon_{{\bf r}}}(t)P_k^{\Upsilon_{{\bf r}}}(t) +(G_k^{\Upsilon_{{\bf r}}}(t))^2}{(P_k^{\Upsilon_{{\bf r}}}(t))^3}.
\end{eqnarray}
Thus
\begin{eqnarray}\label{15eq:nlap}
&&S_k^{\Upsilon_{{\bf r}}}(p,q,r,t) = \nonumber \\
&&D_k^{\Upsilon_{{\bf r}}}(x,p,q,r,t)|_{x^2} - R_k^{\Upsilon_{{\bf r}}}(p,q,r,t) = \nonumber \\
&& \frac{H_k^{\Upsilon_{{\bf r}}}(t)P_k^{\Upsilon_{{\bf r}}}(t) +G_k^{\Upsilon_{{\bf r}}}(t))^2 -([k]_rt -1)P_k^{\Upsilon_{{\bf r}}}(t) - (P_k^{\Upsilon_{{\bf r}}}(t))^2
+ G_k^{\Upsilon_{{\bf r}}}(t)P_k^{\Upsilon_{{\bf r}}}(t)}{(P_k^{\Upsilon_{{\bf r}}}(t))^3} = \nonumber \\
&&
\frac{(G_k^{\Upsilon_{{\bf r}}}(t))^2 + P_k^{\Upsilon_{{\bf r}}}(t)(H_k^{\Upsilon_{{\bf r}}}(t) +G_k^{\Upsilon_{{\bf r}}}(t)-([k]_rt -1) - P_k^{\Upsilon_{{\bf r}}}(t))}{(P_k^{\Upsilon_{{\bf r}}}(t))^3}.
\end{eqnarray}
The exact same sequence of steps work in the other cases so that
we have the following theorems.

\begin{theorem} For all $k \geq 2$,

\begin{eqnarray}
R_k^{\Upsilon_{{\bf r}}}(p,q,r,t) &=& \sum_{n \geq 0} \frac{t^n}{[n]_{p,q}!}
\sum_{(\sg,w) \in \mathcal{U}^{\Upsilon_{{\bf r}}}_{n,k}} q^{\inv{\sg}}p^{\coinv{\sg}}r^{||w||}
\nonumber \\
&=& \frac{([k]_rt -1) + P_k^{\Upsilon_{{\bf r}}}(t) -
G_k^{\Upsilon_{{\bf r}}}(t)}{(P_k^{\Upsilon_{{\bf r}}}(t))^2}
\end{eqnarray}
and
\begin{eqnarray}
S_k^{\Upsilon_{{\bf r}}}(p,q,r,t) &=& \sum_{n \geq 0} \frac{t^n}{[n]_{p,q}!}
\sum_{(\sg,w) \in \mathcal{V}^{\Upsilon_{{\bf r}}}_{n,k}} q^{\inv{\sg}}p^{\coinv{\sg}}r^{||w||}
\nonumber \\
&=& \frac{(G_k^{\Upsilon_{{\bf r}}}(t))^2 + P_k^{\Upsilon_{{\bf r}}}(t)(H_k^{\Upsilon_{{\bf r}}}(t) +G_k^{\Upsilon_{{\bf r}}}(t)-([k]_rt -1) - P_k^{\Upsilon_{{\bf r}}}(t))}{(P_k^{\Upsilon_{{\bf r}}}(t))^3}
\end{eqnarray}
where
\begin{eqnarray*}
P_k^{\Upsilon_{{\bf r}}}(t) = 1+ \sum_{n \geq 1} \frac{(-t)^n}{[n]_{p,q}!}\rbinom{n+k-1}{n}, \\
G_k^{\Upsilon_{{\bf r}}}(t) = \sum_{n \geq 2} (n-1)
\frac{(-t)^n}{[n]_{p,q}!}\rbinom{n+k-1}{n},\ \mbox{and} \\
H_k^{\Upsilon_{{\bf r}}}(t) = - \sum_{n \geq 3}  \binom{n-1}{2}
\frac{(-t)^n}{[n]_{p,q}!}\rbinom{n+k-1}{n}.
\end{eqnarray*}

\end{theorem}

\begin{theorem} For all $k \geq 2$,

\begin{eqnarray}
R_k^{\Upsilon_{{\bf w}}}(p,q,r,t) &=& \sum_{n \geq 0} \frac{t^n}{[n]_{p,q}!}
\sum_{(\sg,w) \in \mathcal{U}^{\Upsilon_{{\bf w}}}_{n,k}} q^{\inv{\sg}}p^{\coinv{\sg}}r^{||w||}
\nonumber \\
&=& \frac{([k]_rt -1) + P_k^{\Upsilon_{{\bf w}}}(t) -
G_k^{\Upsilon_{{\bf w}}}(t)}{(P_k^{\Upsilon_{{\bf w}}}(t))^2}
\end{eqnarray}
and
\begin{eqnarray}
S_k^{\Upsilon_{{\bf w}}}(p,q,r,t) &=& \sum_{n \geq 0} \frac{t^n}{[n]_{p,q}!}
\sum_{(\sg,w) \in \mathcal{V}^{\Upsilon_{{\bf w}}}_{n,k}} q^{\inv{\sg}}p^{\coinv{\sg}}r^{||w||}
\nonumber \\
&=& \frac{(G_k^{\Upsilon_{{\bf w}}}(t))^2 + P_k^{\Upsilon_{{\bf w}}}(t)(H_k^{\Upsilon_{{\bf w}}}(t) +G_k^{\Upsilon_{{\bf w}}}(t)-([k]_rt -1) - P_k^{\Upsilon_{{\bf w}}}(t))}{(P_k^{\Upsilon_{{\bf w}}}(t))^3}
\end{eqnarray}
where
\begin{eqnarray*}
P_k^{\Upsilon_{{\bf w}}}(t) = 1+ \sum_{n \geq 1} \frac{(-t)^n}{[n]_{p,q}!} [k]_{r^n}, \\
G_k^{\Upsilon_{{\bf w}}}(t) = \sum_{n \geq 2} (n-1)
\frac{(-t)^n}{[n]_{p,q}!}[k]_{r^n},\ \mbox{and} \\
H_k^{\Upsilon_{{\bf w}}}(t) = - \sum_{n \geq 3}  \binom{n-1}{2}
\frac{(-t)^n}{[n]_{p,q}!}[k]_{r^n}.
\end{eqnarray*}

\end{theorem}

\begin{theorem} For all $k \geq 2$,

\begin{eqnarray}
R_k^{\Upsilon_{{\bf s}}}(p,q,r,t) &=& \sum_{n \geq 0} \frac{t^n}{[n]_{p,q}!}
\sum_{(\sg,w) \in \mathcal{U}^{\Upsilon_{{\bf s}}}_{n,k}} q^{\inv{\sg}}p^{\coinv{\sg}}r^{||w||}
\nonumber \\
&=& \frac{([k]_rt -1) + P_k^{\Upsilon_{{\bf s}}}(t) -
G_k^{\Upsilon_{{\bf s}}}(t)}{(P_k^{\Upsilon_{{\bf s}}}(t))^2}
\end{eqnarray}
and
\begin{eqnarray}
S_k^{\Upsilon_{{\bf s}}}(p,q,r,t) &=& \sum_{n \geq 0} \frac{t^n}{[n]_{p,q}!}
\sum_{(\sg,w) \in \mathcal{V}^{\Upsilon_{{\bf s}}}_{n,k}} q^{\inv{\sg}}p^{\coinv{\sg}}r^{||w||}
\nonumber \\
&=& \frac{(G_k^{\Upsilon_{{\bf s}}}(t))^2 + P_k^{\Upsilon_{{\bf s}}}(t)(H_k^{\Upsilon_{{\bf s}}}(t) +G_k^{\Upsilon_{{\bf s}}}(t)-([k]_rt -1) - P_k^{\Upsilon_{{\bf s}}}(t))}{(P_k^{\Upsilon_{{\bf s}}}(t))^3}
\end{eqnarray}
where
\begin{eqnarray*}
P_k^{\Upsilon_{{\bf s}}}(t) = 1+ \sum_{n \geq 1} \frac{(-t)^n}{[n]_{p,q}!} \rbinom{k}{n}, \\
G_k^{\Upsilon_{{\bf s}}}(t) = \sum_{n \geq 2} (n-1)
\frac{(-t)^n}{[n]_{p,q}!}\rbinom{k}{n},\ \mbox{and} \\
H_k^{\Upsilon_{{\bf s}}}(t) = - \sum_{n \geq 3}  \binom{n-1}{2}
\frac{(-t)^n}{[n]_{p,q}!}\rbinom{k}{n}.
\end{eqnarray*}

\end{theorem}

\begin{theorem} For all $k \geq 2$,

\begin{eqnarray}
R_k^{\Upsilon_{{\bf d}}}(p,q,1,t) &=& \sum_{n \geq 0} \frac{t^n}{[n]_{p,q}!}
\sum_{(\sg,w) \in \mathcal{U}^{\Upsilon_{{\bf d}}}_{n,k}} q^{\inv{\sg}}p^{\coinv{\sg}}
\nonumber \\
&=& \frac{(kt -1) + P_k^{\Upsilon_{{\bf d}}}(t) -
G_k^{\Upsilon_{{\bf d}}}(t)}{(P_k^{\Upsilon_{{\bf d}}}(t))^2}
\end{eqnarray}
and
\begin{eqnarray}
S_k^{\Upsilon_{{\bf d}}}(p,q,1,t) &=& \sum_{n \geq 0} \frac{t^n}{[n]_{p,q}!}
\sum_{(\sg,w) \in \mathcal{V}^{\Upsilon_{{\bf d}}}_{n,k}} q^{\inv{\sg}}p^{\coinv{\sg}}
\nonumber \\
&=& \frac{(G_k^{\Upsilon_{{\bf d}}}(t))^2 + P_k^{\Upsilon_{{\bf d}}}(t)(H_k^{\Upsilon_{{\bf d}}}(t) +G_k^{\Upsilon_{{\bf d}}}(t)-(kt -1) - P_k^{\Upsilon_{{\bf d}}}(t))}{(P_k^{\Upsilon_{{\bf d}}}(t))^3}
\end{eqnarray}
where
\begin{eqnarray*}
P_k^{\Upsilon_{{\bf d}}}(t) = 1+ \sum_{n \geq 1} \frac{(-t)^n}{[n]_{p,q}!} k(k-1)^{n-1}, \\
G_k^{\Upsilon_{{\bf d}}}(t) = \sum_{n \geq 2} (n-1)
\frac{(-t)^n}{[n]_{p,q}!}k(k-1)^{n-1},\ \mbox{and} \\
H_k^{\Upsilon_{{\bf d}}}(t) = - \sum_{n \geq 3}  \binom{n-1}{2}
\frac{(-t)^n}{[n]_{p,q}!}k(k-1)^{n-1}.
\end{eqnarray*}

\end{theorem}

\section{Numbers involved; bijective questions} The generating functions from the
previous sections allows us to easily compute the initial sequences
of values for these generating functions using any computer algebra
system such as Mathematica or Maple. For example, let
\begin{equation}
A_k^{\Upsilon}(1,1,1,t) = \sum_{n\geq 0} A_{n,k}^{\Upsilon} \frac{t^n}{n!}
\end{equation}
so that $A_{n,k}^{\Upsilon}$ is equal to the number of $(\sg,\ep)
\in C_k \wr S_n$ such that $\Umch{(\sg,\ep)} = 0$.

\subsection{$\Upsilon_{{\bf r}} = \{(1~2,0~0),(1~2,0~1)\}$} For $\Upsilon_{{\bf r}} = \{(1~2,0~0),(1~2,0~1)\}$,
$A_{n,k}^{\Upsilon_{{\bf r}}}$ equals the number of $(\sg,\ep) \in
C_k \wr S_n$ such that $\ris{(\sg,\ep)} = 0$. Table \ref{table:Ankr}
gives initial values of $A_{n,k}^{\Upsilon_{{\bf r}}}$.

\begin{table}[h]
\caption{$A_{n,k}^{\Upsilon_{{\bf r}}}$ for $k,n\leq 5$.}
\[
\begin{array}{|c|c|c|c|c|c|c|}
 \hline       & n=0 & n=1 & n=2 & n=3 & n=4 & n=5\\
\hline k=2  & 1&2&5&16&65&326\\
\hline k=3 &  1&3&12&64&441&3771\\
\hline k=4 &  1&4&22&164&1589&19136\\
\hline k=5 &  1&5&35&335&4180&64876 \\ \hline
\end{array}
\]
\label{table:Ankr}
\end{table}

Several of these sequences appear in \cite{OEIS}.

In fact, we can easily calculate $A_{n,k}^{\Upsilon_{{\bf r}}}$ as a polynomial
in $k$. For example, we have
\begin{eqnarray*}
A_{0,k}^{\Upsilon_{{\bf r}}} &=& 1\\
A_{1,k}^{\Upsilon_{{\bf r}}} &=& k\\
A_{2,k}^{\Upsilon_{{\bf r}}} &=& \frac12 k(3k-1)\\
A_{3,k}^{\Upsilon_{{\bf r}}} &=& \frac16 k(19k^2-15k+2)\\
A_{4,k}^{\Upsilon_{{\bf r}}} &=& \frac{1}{24} k(211k^3 -270k^2+89k-6)\\
A_{5,k}^{\Upsilon_{{\bf r}}} &=& \frac{1}{120}k(3651k^4-6490k^3+3585k^2-650k+24)\\
\end{eqnarray*}
We point out that $A_{2,k}^{\Upsilon_{{\bf r}}}$ forms the familiar sequence of pentagonal numbers (A000326 in \cite{OEIS}). Other previously documented sequences appearing in Table \ref{table:Ankr} include the structured octagonal anti-prism numbers (A100184 in \cite{OEIS}) for $A_{3,k}^{\Upsilon_{{\bf r}}}$; as well as $A_{n,2}^{\Upsilon_{{\bf r}}}$ (A000522 in \cite{OEIS}), for which there are many known combinatorial interpretations, including the total number of arrangements of all subsets of $[n]$.

We conjecture that for $n \geq 1$ and $k \geq 2$, $A_{n,k}^{\Upsilon_{{\bf r}}}$ is always of the form
$\frac{1}{n!}kP_n(k)$ where $P_n(k)$ is a polynomial of degree $n-1$ whose leading coefficient is positive and such that signs of the remaining coefficients alternate. Now we can prove that $A_{n,k}^{\Upsilon_{{\bf r}}}$ is always of the form
$\frac{1}{n!}kP_n(k)$ where $P_n(k)$ is a polynomial of degree $n-1$ and
the term of degree 1 in $k$ is $(-1)^{n-1}(n-1)!$.
That is, for any $k \geq 2$,
if we set $p =q =r =1$ and $x =0$ in (\ref{ris1}), we see that
\begin{equation}\label{rpoly1}
n! \bar{\Gamma}(h_n) = A_{n,k}^{\Upsilon_{{\bf r}}}
\end{equation}
where
\begin{equation}\label{rpoly2}
\bar{\Gamma}(e_n) = \frac{\binom{n+k-1}{n}}{n!} = \frac{(k)\uparrow_n}{(n!)^2}.
\end{equation}
Here we let $(q)\uparrow_0 =1$ and $(q)\uparrow_n =q(q+1) \ldots (q+n-1)$
for $n \geq 1$. But then
\begin{eqnarray}\label{rpoly3}
n!\bar{\Gamma}(h_n) &=& n!\sum_{\mu \vdash n} (-1)^{n-\ell(\mu)}
\bar{\Gamma}(e_{\mu})  \nonumber \\
&=& \frac{1}{n!} \sum_{\mu \vdash n} (-1)^{n-\ell(\mu)}
\binom{n}{\mu_1, \ldots, \mu_{\ell(\mu)}}^2 \prod_{i=1}^{\ell(\mu)}
(k)\uparrow_{\mu_i}.
\end{eqnarray}
It is easy to see that the right hand side of (\ref{rpoly3}) is a
polynomial of degree $n$ and the lowest degree term comes from the
term $(-1)^{n-1} k(k+1)\ldots (k+n-1)$ corresponding to $\mu = (n)$
which is of the form $(-1)^{n-1} (n-1)!k +O(k^2)$.

\subsection{$\Upsilon_{{\bf s}} = \{(1~2,0~1)\}$} For $\Upsilon_{{\bf s}} = \{(1~2,0~1)\}$,
$A_{n,k}^{\Upsilon_{{\bf s}}}$ equals the number of $(\sg,\ep) \in
C_k \wr S_n$ such that $\sris{(\sg,\ep)} = 0$. Table
\ref{table:Anks} gives initial values of $A_{n,k}^{\Upsilon_{{\bf
s}}}$.

\begin{table}[h]
\caption{$A_{n,k}^{\Upsilon_{{\bf s}}}$ for $k,n\leq 5$.}
\[
\begin{array}{|c|c|c|c|c|c|c|}
 \hline       & n=0 & n=1 & n=2 & n=3 & n=4 & n=5\\
\hline k=2  & 1&2&7&36&246&2100\\
\hline k=3 &  1&3&15&109&1050&12630\\
\hline k=4 &  1&4&26&244&3031&47000\\
\hline k=5 &  1&5&40&460&6995&132751\\
\hline
\end{array}
\]
\label{table:Anks}
\end{table}

In fact, we can easily calculate $A_{n,k}^{\Upsilon_{{\bf s}}}$ as a polynomial in $k$. For example, we have
\begin{eqnarray*}
A_{0,k}^{\Upsilon_{{\bf s}}} &=& 1\\
A_{1,k}^{\Upsilon_{{\bf s}}} &=& k\\
A_{2,k}^{\Upsilon_{{\bf s}}} &=& \frac12 k(3k+1)\\
A_{3,k}^{\Upsilon_{{\bf s}}} &=& \frac16 k(19k^2+15k+2)\\
A_{4,k}^{\Upsilon_{{\bf s}}} &=& \frac{1}{24} k(211k^3+270k^2+89k+6)\\
A_{5,k}^{\Upsilon_{{\bf s}}} &=& \frac{1}{120}k(3651k^4+6490k^3+3585k^2+650k+24)\\
\end{eqnarray*}
We point out that $A_{2,k}^{\Upsilon_{{\bf s}}}$ forms the familiar sequence of the second pentagonal numbers (A005449 in \cite{OEIS}). None of the other rows or columns in Table \ref{table:Anks} matched any previously known sequences in \cite{OEIS}.

We conjecture that for $n \geq 1$ and $k \geq 2$, $A_{n,k}^{\Upsilon_{{\bf s}}}$ is always of the form
$\frac{1}{n!}kR_n(k)$ where $R_n(k)$ is a polynomial of degree $n-1$ with positive coefficients. In fact, we see that the coefficients of $P_{n,k}$ and $R_{n,k}$ are the same up to a sign for all $n$. This we can prove.
That is, for any $k \geq 2$,
if we set $p =q =r =1$ and $x =0$ in (\ref{sris1}), we see that
\begin{equation}\label{spoly1}
n! \bar{\Gamma}_s(h_n) = A_{n,k}^{\Upsilon_{{\bf s}}}
\end{equation}
where
\begin{equation}\label{spoly2}
\bar{\Gamma}_s(e_n) = \frac{\binom{k}{n}}{n!} = \frac{(k)\downarrow_n}{(n!)^2}.
\end{equation}
Here we let $(q)\downarrow_0 =1$ and $(q)\downarrow_n =q(q-1) \ldots (q-n+1)$
for $n \geq 1$. But then
\begin{eqnarray}\label{spoly3}
n!\bar{\Gamma}_s(h_n) &=& n!\sum_{\mu \vdash n} (-1)^{n-\ell(\mu)}
\bar{\Gamma}_s(e_{\mu})  \nonumber \\
&=& \frac{1}{n!} \sum_{\mu \vdash n} (-1)^{n-\ell(\mu)}
\binom{n}{\mu_1, \ldots \mu_{\ell(\mu)}}^2 \prod_{i=1}^{\ell(\mu)} (k)\downarrow_{\mu_i}.
\end{eqnarray}
Since for any $n \geq 1$, $(k)\downarrow_n = (-1)^n(k)\uparrow_n$,
it is easy to see that the right hand side of (\ref{spoly3}) is
obtained from the right hand side of (\ref{rpoly3}) by replacing $k$
by $-k$ and multiplying by $(-1)^n$. Thus the conjecture that
$R_n(k)$ has positive coefficients is equivalent to our conjecture
that the signs of the coefficients of $P_n(k)$ alternate.

\section{$\Upsilon_{{\bf w}} = \{(1~2,0~0)\}$} For $\Upsilon_{{\bf w}} = \{(1~2,0~0)\}$,
$A_{n,k}^{\Upsilon_{{\bf w}}}$ equals the number of $(\sg,\ep) \in
C_k \wr S_n$ such that $\wris{(\sg,\ep)} = 0$. Table
\ref{table:Ankw} gives initial values of $A_{n,k}^{\Upsilon_{{\bf
w}}}$.

\begin{table}[h]
\caption{$A_{n,k}^{\Upsilon_{{\bf w}}}$ for $k,n\leq 5$.}
\[
\begin{array}{|c|c|c|c|c|c|c|}
 \hline       & n=0 & n=1 & n=2 & n=3 & n=4 & n=5\\
\hline k=2  & 1&2&6&26&150&1082\\
\hline k=3 &  1&3&15&111&1095&13503\\
\hline k=4 &  1&4&28&292&4060&70564\\
\hline k=5 &  1&5&45&605&10845&243005\\
\hline
\end{array}
\]
\label{table:Ankw}
\end{table}

In fact, we can easily calculate $A_{n,k}^{\Upsilon_{{\bf w}}}$ as a polynomial
in $k$. For example, we have
\begin{eqnarray*}
A_{0,k}^{\Upsilon_{{\bf w}}} &=& 1\\
A_{1,k}^{\Upsilon_{{\bf w}}} &=& k\\
A_{2,k}^{\Upsilon_{{\bf w}}} &=& k(2k-1)\\
A_{3,k}^{\Upsilon_{{\bf w}}} &=& k(6k^2-6k+1)\\
A_{4,k}^{\Upsilon_{{\bf w}}} &=& k(24k^3 -36k^2+14k-1)\\
A_{5,k}^{\Upsilon_{{\bf w}}} &=& k(120k^4-240k^3+150k^2-30k+1)\\
\end{eqnarray*}
We point out that $A_{2,k}^{\Upsilon_{{\bf w}}}$ forms the familiar sequence of hexagonal numbers (A000384 in \cite{OEIS}). Additionally, $A_{n,2}^{\Upsilon_{{\bf w}}}$ matches the sequence counting the number of necklaces on set of labeled beads (A000629 in \cite{OEIS}). In fact, in this case we can
give a completely combinatorial interpretation of
$A_{n,k}^{\Upsilon_{{\bf w}}}$.  Let $OSetpn(n)$ denote
the set of ordered set partitions of $\{1, \ldots, n\}$. For
any set partition $\pi \in OSetpn(n)$, let $\ell(\pi)$ denote
the number of parts of $\pi$. Then we claim that
\begin{equation}\label{wpoly0}
A_{n,k}^{\Upsilon_{{\bf w}}} = \sum_{\pi \in OSetpn(n)} (-1)^{n -
\ell(\pi)} k^{\ell(\pi)}
\end{equation}
so that the coefficient of $k^j$ in $A_{n,k}^{\Upsilon_{{\bf w}}}$
is equal to $(-1)^{n-j}j!S_{n,j}$ where $S_{n,j}$ is the Stirling
number of the second kind which is the number of set partitions of
$\{1, \ldots, n\}$ into $j$ parts. That is, for any $k \geq 2$, if
we set $p =q =r =1$ and $x =0$ in (\ref{wris1}), we see that
\begin{equation}\label{wpoly1}
n! \bar{\Gamma}_w(h_n) = A_{n,k}^{\Upsilon_{{\bf w}}}
\end{equation}
where
\begin{equation}\label{wpoly2}
\bar{\Gamma}_w(e_n) = \frac{k}{n!}.
\end{equation}
But then
\begin{eqnarray}\label{wpoly3}
n!\bar{\Gamma}_s(h_n) &=& n!\sum_{\mu \vdash n} (-1)^{n-\ell(\mu)}
\bar{\Gamma}_s(e_{\mu})  \nonumber \\
&=& n!\sum_{\mu \vdash n} (-1)^{n-\ell(\mu)}
\sum_{(b_1, \ldots, b_{\ell(\mu)}) \in \mathcal{B}_{\mu,n}}
\prod_{i=1}^{\ell(\mu)} \frac{k}{b_i!} \nonumber \\
&=& \sum_{\mu \vdash n} (-1)^{n-\ell(\mu)}
\sum_{(b_1, \ldots, b_{\ell(\mu)}) \in \mathcal{B}_{\mu,n}} \binom{n}{b_1, \ldots, b_{\ell(\mu)}} k^{\ell(\mu)}.
\end{eqnarray}
Since $\binom{n}{b_1, \ldots, b_{\ell(\mu)}}$ counts the number of
ordered set partitions $\pi = (\pi_1, \ldots, \pi_{\ell(\mu)})$ such
that $|\pi_j| = b_j$, it is easy to see that the right hand side of
(\ref{wpoly3}) equals the right hand side of (\ref{wpoly0}).

\subsection{$\Upsilon_{{\bf d}} = \{(1~2,0~1),(1~2,1~0)\}$} Table
\ref{table:Ankd} gives initial values of $A_{n,k}^{\Upsilon_{{\bf
d}}}$.

\begin{table}[h]
\caption{$A_{n,k}^{\Upsilon_{{\bf d}}}$ for $k,n\leq 5$.}
\[
\begin{array}{|c|c|c|c|c|c|c|}
 \hline       & n=0 & n=1 & n=2 & n=3 & n=4 & n=5\\
\hline $k=2$  & 1&2&6&26&150&1082\\
\hline $k=3$ &  1&3&12&66&480&4368\\
\hline $k=4$ &  1&4&20&132&1140&12324\\
\hline $k=5$ &  1&5&30&230&2280&28280\\
\hline
\end{array}
\]
\label{table:Ankd}
\end{table}

In fact, we can easily calculate $A_{n,k}^{\Upsilon_{{\bf d}}}$ as a polynomial in $k$. For example, we have
\begin{eqnarray*}
A_{0,k}^{\Upsilon_{{\bf d}}} &=& 1\\
A_{1,k}^{\Upsilon_{{\bf d}}} &=& k\\
A_{2,k}^{\Upsilon_{{\bf d}}} &=& k^2+k\\
A_{3,k}^{\Upsilon_{{\bf d}}} &=& k^3+4k^2+k\\
A_{4,k}^{\Upsilon_{{\bf d}}} &=& k^4+11k^3+11k^2+k\\
A_{5,k}^{\Upsilon_{{\bf d}}} &=& k^5+26k^4+66k^3+26k^2+k\\
\end{eqnarray*}
In this case, we shall show that $A_{n,k}^{\Upsilon_{{\bf d}}}$
is just the Eulerian polynomial
\begin{equation}\label{dpoly0}
A_{n,k}^{\Upsilon_{{\bf d}}} = \sum_{\sg \in S_n} x^{\des{\sg}+1}.
\end{equation}
That is,
for any $k \geq 2$,
if we set $p =q =r =1$ and $x =0$ in (\ref{Uris1}), we see that
\begin{equation}\label{dpoly1}
n! \bar{\Gamma}_U(h_n) = A_{n,k}^{\Upsilon_{{\bf d}}}
\end{equation}
where
\begin{equation}\label{dpoly2}
\bar{\Gamma}_U(e_n) = \frac{k(k-1)^{n-1}}{n!}.
\end{equation}
But then
\begin{eqnarray}\label{dpoly3}
n!\bar{\Gamma}_U(h_n) &=& n!\sum_{\mu \vdash n} (-1)^{n-\ell(\mu)}
\bar{\Gamma}_U(e_{\mu})  \nonumber \\
&=& n!\sum_{\mu \vdash n} (-1)^{n-\ell(\mu)}
\sum_{(b_1, \ldots, b_{\ell(\mu)}) \in \mathcal{B}_{\mu,n}}
\prod_{i=1}^{\ell(\mu)} \frac{k(k-1)^{b_i-1}}{b_i!} \nonumber \\
&=& \sum_{\mu \vdash n} \sum_{(b_1, \ldots, b_{\ell(\mu)}) \in
\mathcal{B}_{\mu,n}} \binom{n}{b_1, \ldots, b_{\ell(\mu)}}
\prod_{i=1}^{\ell(\mu)} k(1-k)^{b_i-1}.
\end{eqnarray}
Next we want to give a combinatorial interpretation to
(\ref{dpoly3}). For any brick tabloid $T= (b_1, \ldots,
b_{\ell(\mu)}) \in \mathcal{B}_{\mu,n}$, we can interpret
$\binom{n}{b_1, \ldots, b_{\ell(\mu)}}$ as the set of all fillings
of $T$ with a permutation $\sg \in S_{n}$ such that $\sg$ is
increasing in each brick.  We then interpret
$\prod_{j=1}^{\ell(\mu)} k(1-k)^{b_j-1}$ as all ways of picking a
label of the cells of each brick except the final cell with either
an $1$ or a $-k$ and letting the label of the last cell of each
brick be $k$. We let $\mathcal{D}_{n}$ denote the set of all filled
labelled brick tabloids that arise in this way.  Thus a $C \in
\mathcal{D}_{n}$ consists of a brick tabloid $T$, a permutation $\sg
\in S_{n}$ and a labelling $L$ of the cells of $T$ with elements
from $\{k,-k,1\}$ such that
\begin{enumerate}
\item $\sg$ is strictly increasing in each brick,
 \item the final cell of each
brick is labelled with k, and
\item each cell which is not a final
cell of a brick is labelled with $1$  or $-k$.
\end{enumerate}
We then define the weight $w(C)$ of $C$ to be the product of all
the $k$ labels in $L$ and the sign $sgn(C)$ of $C$ to be
the product of all the $-1$ labels in $L$. For example,
if $n =12$, $k=4$, and $T =(4,3,3,2)$, then Figure \ref{figure:ddes1}
pictures such a composite object $C \in \mathcal{D}_{12}$ where
$w(C) = k^7$ and $sgn(C) =-1$.

Thus
\begin{equation}\label{dpoly4}
n!\bar{\Gamma}_U(h_{n}) = \sum_{C \in \mathcal{D}_{n}}
sgn(C) w(C).
\end{equation}

\fig{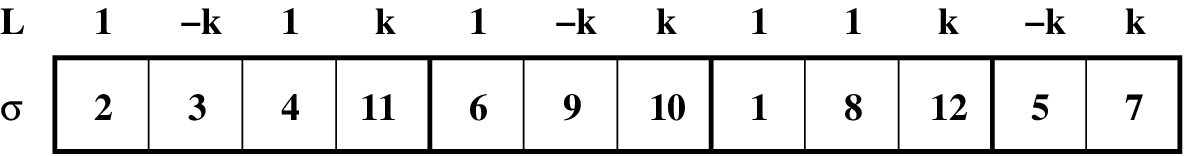}{A composite object $C \in \mathcal{D}_{12}$.}

Next we define a weight-preserving sign-reversing involution
$J:\mathcal{D}_{n} \rightarrow \mathcal{D}_{n}$.  To define
$J(C)$, we scan the cells of $C =(T,\sg, L)$ from left  to right
looking for the leftmost cell $t$ such that either (i) $t$ is
labelled with $-k$ or (ii) $t$ is at the end of a brick $b_j$ and
the brick $b_{j+1}$ immediately following $b_j$ has the property
that $\sg$ is strictly increasing in all the cells corresponding
to $b_j$ and $b_{j+1}$.   In case (i), $J(C)
=(T',\sg', L')$ where $T'$ is the result of  replacing the brick
$b$ in $T$ containing $t$ by two bricks $b^*$ and $b^{**}$ where
$b^*$ contains the cell $t$ plus all the cells in $b$ to the left
of $t$ and $b^{**}$ contains all the cells of $b$ to the right of
$t$, $\sg =\sg'$, and $L'$ is the labelling that results
from $L$ by changing the label of cell $t$ from $-k$ to $k$. In
case (ii), $J(C) =(T',\sg', L')$ where $T'$ is the result of
replacing the bricks $b_j$ and $b_{j+1}$ in $T$ by a single brick
$b$, $\sg =\sg'$, and $L'$ is the labelling that results
from $L$ by changing the label of cell $t$ from $k$ to $-k$. If
neither case (i) or case (ii) applies, then we let $J(C) =C$. For
example, if $C$ is the element of $\mathcal{D}_{12}$ pictured in
Figure \ref{figure:ddes1}, then $J(C)$ is pictured in Figure
\ref{figure:ddes2}.

\fig{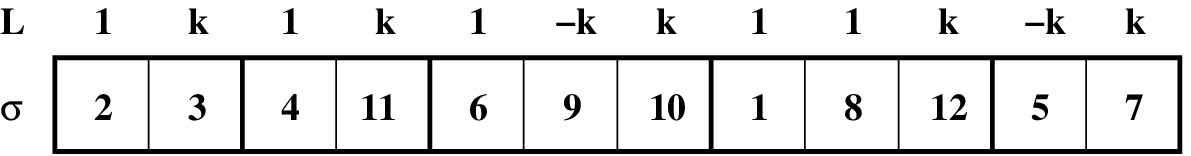}{$J(C)$ for $C$ in Figure \ref{figure:ddes1}.}

It is easy to see that $J$ is a weight-preserving sign-reversing
involution and hence $J$ shows that
\begin{equation}\label{dpoly5}
n!\bar{\Gamma}_U(h_n) = \sum_{C \in \mathcal{D}_{n},J(C) = C}
sgn(C) w(C).
\end{equation}

Thus we must examine the fixed points $C = (T,\sg,L)$ of $J$.
First there can be no $-k$ labels in $L$ so that $sgn(C) =1$.
Moreover,  if $b_j$ and $b_{j+1}$ are two consecutive bricks in $T$
and $t$ is the last cell of $b_j$, then it can not be the case that
$\sg_{t} < \sg_{t+1}$ since otherwise we
could combine $b_j$ and $b_{j+1}$. For any such fixed point, we
associate an element $(\sg,w) \in C_k \wr S_{n}$. For example, a
fixed point of $I$ is pictured in Figure \ref{figure:ddes3} where
\begin{equation*}
\sg = 2~3~4~11~6~9~10~1~8~12~5~7.
\end{equation*}
It follows that if cell $t$ is at the end of a brick which is not
the last brick, then $\sg_t > \sg_{t+1}$. However if $v$ is a cell
which is not at the end of a brick, then our definitions force
$\sg_{v} < \sg_{v+1}$.  Since each such cell $v$ must be labelled
with an $1$, it follows that $sgn(C)w(C) = k^{des(\sg)+1}$ where the
$+1$ comes from the fact that the last cell of the last brick is
also labeled with $k$.
  Vice versa, if
$\sg \in S_{n}$, then we can create a fixed point $C
=(T,\sg,L)$ by having the bricks in $T$ end at cells of the form
$t$ where $\sg_t > \sg_{t+1}$ and labeling each such cell with $k$,
labeling the last cell with $k$,  and labelling the remaining cells
with $1$. Thus we have shown that
\begin{equation*}
n!\bar{\Gamma}_U(h_{n}) = \sum_{\sg \in S_n} k^{\des{\sg}+1}
\end{equation*}
as desired.

\fig{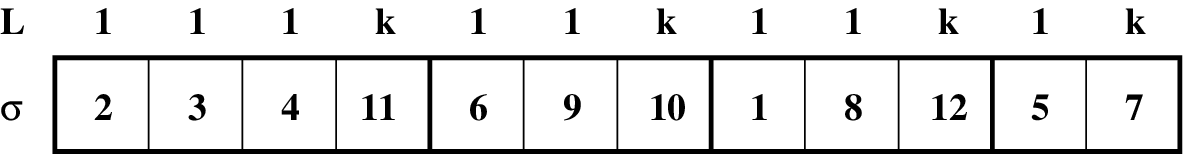}{A fixed point of $J$.}

\subsection{$U_{n,k,{\bf a}}=
|\mathcal{U}^{\Upsilon_{{\bf a}}}_{n,k}|$ and $V_{n,k,{\bf a}}
=|\mathcal{V}^{\Upsilon_{{\bf a}}}_{n,k}|$} We have computed similar
tables for $U_{n,k,{\bf a}}= |\mathcal{U}^{\Upsilon_{{\bf
a}}}_{n,k}|$ and $V_{n,k,{\bf a}} =|\mathcal{V}^{\Upsilon_{{\bf
a}}}_{n,k}|$ using our formulas for the generating functions
$R_k^{\Upsilon_{{\bf a}}}(p,q,r,t)$ and $S_k^{\Upsilon_{{\bf
a}}}(p,q,r,t)$. Table \ref{table:Unkr} gives initial values of
$U_{n,k,{\bf r}}$, which counts the number of $(\sg,\ep) \in C_k \wr
S_n$ such that $\Ris((\sg,\ep)) = \{s,s+1\}$ for some $1 \leq s \leq
n-2$.

\begin{table}[h]
\caption{$U_{n,k,{\bf r}}$ for $n\leq 7$, $k\leq 5$.}
\[
\begin{array}{|c|c|c|c|c|c|}
\hline       & n=3 & n=4 & n=5 & n=6 & n=7 \\
\hline k=2  & 4&54&538&5064&48900\\
\hline k=3 &  10&210&3363&52056&838542\\
\hline k=4 &  20&570&12568&270328&6083712\\
\hline k=5 &  35&1260&35328&973840&28127160\\
\hline
\end{array}
\]
\label{table:Unkr}
\end{table}

In fact, we can easily calculate $U_{n,k,{\bf r}}$ as a polynomial in $k$. For example, we have
\begin{eqnarray*}
U_{3,k,{\bf r}} &=& \frac{1}{6}k(k+1)(k+2)\\
U_{4,k,{\bf r}} &=& \frac{1}{4}k(k+1)(k+2)(5k-1)\\
U_{5,k,{\bf r}} &=& \frac{1}{120}k(k+1)(k+2)(903k^2-479k+36)\\
U_{6,k,{\bf r}} &=& \frac{1}{45}k(k+1)(k+2)(2032k^3-1896k^2+419k-15)\\
U_{7,k,{\bf r}} &=& \frac{1}{1680}k(k+1)(k+2)(482031k^4-662450k^3 +268653k^2-32554k+600)
\end{eqnarray*}
Thus we conjecture that for $n \geq 3$ and $k \geq 2$, $U_{n,k,{\bf r}}$ is always of the form
$\frac{1}{n!}k(k+1)(k+2)U_{n,1}(k)$ where $U_{n,{\bf r}}(k)$ is a polynomial of degree $n-3$ whose leading coefficient is positive and such that signs of the remaining coefficients alternate.

Furthermore, we point out that $U_{3,k}^{\Upsilon_{{\bf r}}}$ forms the familiar sequence of tetrahedral numbers (A000292 in \cite{OEIS}). None of the other rows or columns in Table \ref{table:Unkr} matched any previously known sequence in \cite{OEIS}.

Table \ref{table:Unks} gives initial values of $U_{n,k,{\bf s}}$, which counts the number of $(\sg,\ep) \in C_k \wr S_n$ such that $\SRis((\sg,\ep)) = \{s,s+1\}$ for some $1 \leq s \leq n-2$.

\begin{table}[h]
\caption{$U_{n,k,{\bf s}}$ for $n\leq 7$, $k\leq 5$.}
\[
\begin{array}{|c|c|c|c|c|c|}
\hline       & n=3 & n=4 & n=5 & n=6 & n=7 \\
\hline k=2  & 0&0&0&0&0\\
\hline k=3 &  1&24&480&9760&212310\\
\hline k=4 &  4&126&3280&86440&2431800\\
\hline k=5 &  10&390&12503&404688&13962690\\
\hline
\end{array}
\]
\label{table:Unks}
\end{table}

In fact, we can easily calculate $U_{n,k,{\bf s}}$ as a polynomial in $k$. For example, we have
\begin{eqnarray*}
U_{3,k,{\bf s}} &=& \frac{1}{6}k(k-1)(k-2)\\
U_{4,k,{\bf s}} &=& \frac{1}{4}k(k-1)(k-2)(5k+1)\\
U_{5,k,{\bf s}} &=& \frac{1}{120}k(k-1)(k-2)(903k^2+479k+36)\\
U_{6,k,{\bf s}} &=& \frac{1}{45}k(k-1)(k-2)(2032k^3+1896k^2+419k+15)\\
U_{7,k,{\bf s}} &=& \frac{1}{1680}k(k-1)(k-2)(482031k^4+662450k^3 +268653k^2+32554k+600)
\end{eqnarray*}
Thus we conjecture that for $n \geq 3$ and $k \geq 2$, $U_{n,k,{\bf s}}$ is always of the form $\frac{1}{n!}k(k+1)(k+2)U_{n,3}(k)$ where $U_{n,{\bf s}}(k)$ is a polynomial of degree $n-3$ with positive coefficients. Moreover, we conjecture that the coefficients of $U_{n,{\bf s}}(k)$ and $U_{n,{\bf r}}(k)$ are the same up to a sign for $n \geq 3$. None of the rows or columns in Table \ref{table:Unks} matched any previously known non-trivial sequence in \cite{OEIS}.

Table \ref{table:Unkw} gives initial values of $U_{n,k, {\bf w}}$, which counts the number of $(\sg,\ep) \in C_k \wr S_n$ such that $\WRis((\sg,\ep)) = \{s,s+1\}$ for some $1 \leq s \leq n-2$.

\begin{table}[h]
\caption{$U_{n,k,{\bf w}}$ for $n\leq 7$, $k\leq 5$.}
\[
\begin{array}{|c|c|c|c|c|c|}
\hline       & n=3 & n=4 & n=5 & n=6 & n=7 \\
\hline k=2  & 2&28&326&3896&50186\\
\hline k=3 &  3&66&1269&25512&556683\\
\hline k=4 &  4&120&3212&90480&2773140\\
\hline k=5 &  5&190&1303&235880&9303725\\
\hline
\end{array}
\]
\label{table:Unkw}
\end{table}

In fact, we can easily calculate $U_{n,k,{\bf w}}$ as a polynomial in $k$. For example, we have
\begin{eqnarray*}
U_{3,k,{\bf w}} &=& k\\
U_{4,k,{\bf w}} &=& k(8k-2)\\
U_{5,k,{\bf w}} &=& k(60k^2-40k+3)\\
U_{6,k,{\bf w}} &=& 4k(120k^3-135k^2+34k-1)\\
U_{7,k,{\bf w}} &=& k(4200k^4-6720k^3+3108k^2-392k+5)
\end{eqnarray*}
Thus we conjecture that for $n \geq 3$ and $k \geq 2$, $U_{n,k,{\bf w}}$ is always of the form $kU_{n,{\bf w}}(k)$ where $U_{n,2}(k)$ is a polynomial of degree $n-3$ whose leading coefficient is positive and is such that remaining coefficients alternate in sign.

Furthermore, we point out that $U_{4,k}^{\Upsilon_{{\bf w}}}$ forms the sequence of alternating hexagonal numbers (A014635 in \cite{OEIS}). None of the other rows or columns in Table \ref{table:Unkw} matched any previously known sequence in \cite{OEIS}.

Table \ref{table:Unkd} gives initial values of $U_{n,k, {\bf d}}$, which counts the number of $(\sg,\ep) \in C_k \wr S_n$ such that for some $1 \leq s \leq n-2$, $i$ is a start of $\Upsilon_{{\bf d}}$-match if and only if $i \in \{s,s+1\}$.

\begin{table}[h]
\caption{$U_{n,k,{\bf d}}$ for $n\leq 7$, $k\leq 5$.}
\[
\begin{array}{|c|c|c|c|c|c|}
\hline       & n=3 & n=4 & n=5 & n=6 & n=7 \\
\hline k=2  & 2&28&326&3896&50186\\
\hline k=3 &  12&240&3744&58080&958560\\
\hline k=4 &  36&936&18252&345168&6860916\\
\hline k=5 &  80&2560&58840&1329920&30723200\\
\hline
\end{array}
\]
\label{table:Unkd}
\end{table}

In fact, we can easily calculate $U_{n,k,{\bf d}}$ as a polynomial in $k$. For example, we have
\begin{eqnarray*}
U_{3,k,{\bf d}} &=& k(k-1)^2\\
U_{4,k,{\bf d}} &=& 2k(k-1)^2(3k+1)\\
U_{5,k,{\bf d}} &=& k(k-1)^2(23k^2+34k+3)\\
U_{6,k,{\bf d}} &=& 6k(k-1)^2(18k^3+70k^2+31k+1)\\
U_{7,k,{\bf d}} &=& k(k-1)^2(201k^4+1660k^3+1962k^2+372k+5)
\end{eqnarray*}
Thus we conjecture that for $n \geq 3$ and $k \geq 2$, $U_{n,k,{\bf d}}$ is always of the form $k(k-1)^2U_{n,{\bf d}}(k)$ where $U_{n,{\bf d}}(k)$ is a polynomial of degree $n-3$ with positive coefficients. None of the rows or columns in Table \ref{table:Unkd} matched any non-trivial sequence in \cite{OEIS}.

We shall only give polynomial expressions for $V_{k,n,{\bf a}} =|\mathcal{V}_{n,k}^{\Upsilon_{{\bf a}}}|$ for ${\bf a} \in \{{\bf r},{\bf w},{\bf s},{\bf d}\}$ and $n =4,5,6,7$. Note that by definition $V_{n,k,{\bf a}} =0$ for $n =1,2,3$.

For $V_{k,n,{\bf r}}$,  we have the following initial polynomials.

\begin{eqnarray*}
V_{4,k,{\bf r}} &=& \frac{1}{24}k(k+1)(35k^2+31k-6)\\
V_{5,k,{\bf r}} &=& \frac{1}{120}k(k+1)(2253k^3+1277k^2-1022k+72)\\
V_{6,k,{\bf r}} &=& \frac{1}{72}k(k+1)(12781k^4+2336k^3-8911k^2+2146k-72)\\
V_{7,k,{\bf r}} &=& \frac{1}{2520}k(k+1)(3828237k^5-943444k^4-3213331k^3 \\
&& \hspace{45pt} +1679386k^2-207048k+3600)
\end{eqnarray*}
Thus we conjecture that for $n \geq 4$ and $k \geq 2$, $V_{n,k,{\bf r}}$ is always of the form $k(k+1)V_{n,{\bf r}}(k)$ where $V_{n,{\bf r}}(k)$ is a polynomial of degree $n-2$. Note that this is first example where we did not obtain polynomials whose coefficients are either positive or whose coefficients alternate in sign.

However, we still seem to have a type of reciprocity between $V_{n,k,{\bf r}}$ and $V_{n,k,{\bf s}}$. That is, we have the following initial polynomials.
\begin{eqnarray*}
V_{4,k,{\bf s}} &=& \frac{1}{24}k(k-1)(35k^2-31k-6)\\
V_{5,k,{\bf s}} &=& \frac{1}{120}k(k-1)(2253k^3-1277k^2-1022k-72)\\
V_{6,k,{\bf s}} &=& \frac{1}{72}k(k-1)(12781k^4-2336k^3-8911k^2-2146k-72)\\
V_{7,k,{\bf s}} &=& \frac{1}{2520}k(k-1)(3828237k^5+943444k^4-3213331k^3 \\
&& \hspace{45pt} -1679386k^2-207048k-3600)
\end{eqnarray*}
Thus we conjecture that for $n \geq 4$ and $k \geq 2$, $V_{n,k,{\bf s}}$ is always of the form $k(k-1)V_{n,{\bf s}}(k)$ where $V_{n,{\bf s}}(k)$ is a polynomial of degree $n-2$. Moreover we conjecture that the the absolute value of the coefficients in $V_{n,{\bf s}}(k)$ and $V_{n,{\bf r}}(k)$ are the same.

For $V_{k,n,{\bf w}}$,  we have the following initial polynomials.
\begin{eqnarray*}
V_{4,k,{\bf w}} &=& k(6k-1)\\
V_{5,k,{\bf w}} &=& k(90k^2-50k+3)\\
V_{6,k,{\bf w}} &=& 2k(5050k^3-5040k^2+118k-3)\\
V_{7,k,{\bf w}} &=& 2k(6300k^4-9240k^3+3864k^2-434k+5)
\end{eqnarray*}
Thus we conjecture that for $n \geq 4$ and $k \geq 2$, $V_{n,k,{\bf w}}$ is always of the form $kV_{n,{\bf w}}(k)$ where $V_{n,{\bf w}}(k)$ is a polynomial of degree $n-3$ whose leading coefficients is positive and where the signs of the remaining coefficients alternate.

For $V_{k,n,{\bf d}}$,  we have the following initial polynomials.
\begin{eqnarray*}
V_{4,k,{\bf d}} &=& k(k-1)^2(5k+1)\\
V_{5,k,{\bf d}} &=& k(k-1)^2(43k^2+44k+3)\\
V_{6,k,{\bf d}} &=& k(k-1)(230k^3+626k^2+218k+6)\\
V_{7,k,{\bf d}} &=& k(k-1)^2(990k^4+5588k^3+5184k^2+838k+10)
\end{eqnarray*}
Thus we conjecture that for $n \geq 4$ and $k \geq 2$, $V_{n,k,{\bf d}}$ is always of the form $k(k-1)^2V_{n,{\bf d}}(k)$ where $V_{n,{\bf d}}(k)$ is a polynomial of degree $n-3$ whose coefficients are positive.

Note that $U_{n,k,{\bf w}}$ and $V_{n,k,{\bf w}}$ both make sense even in the case where $k=1$. That is, $U_{n,1,{\bf w}}$ equals the number of $\sg \in S_n$ such that $\Ris(\sg) = \{s,s+1\}$ for some $1 \leq s \leq n-2$ and $V_{n,1,{\bf w}}$ equals the number of $\sg \in S_n$ such that $\Ris(\sg) = \{i,j\}$ where $i +2 \leq j$.  Table \ref{table:UVforw} gives these values for small $n$.

\begin{table}[h]
\caption{$U_{n,1,{\bf w}}$ and $V_{n,1,{\bf w}}$ for $n\leq 7$.}
\[
\begin{array}{|c|c|c|c|c|c|}
\hline       & n=3 & n=4 & n=5 & n=6 & n=7 \\
\hline U_{n,1,{\bf w}} & 1&6&23&72&201\\
\hline V_{n,1,{\bf w}} & 0&5&43&230&990\\
\hline
\end{array}
\]
\label{table:UVforw}
\end{table}

One would have expected that generating functions for $U_{n,1,{\bf
w}}$ and $V_{n,1,{\bf w}}$ would have appeared before, but the
sequence for $U_{n,1,{\bf w}}$ appears in OEIS~\cite{OEIS} but not
with our interpretation and the sequence for $V_{n,k,{\bf w}}$ does
not even appear in OEIS~\cite{OEIS} before our work.

\section{Further research}

An obvious direction of research is considering matching conditions
on $C_k \wr S_n$ of length 3 or more and deriving
avoidance/distribution formulas similar to those derived in this
paper.  Another obvious direction of research is to look at distributions
of bi-occurrences of patterns in $C_k \wr S_n$.
One can also consider $k$-tuples of words from a fixed
finite alphabet with the obvious
extension of our matching and occurrence conditions. All of these
topics will be studied in subsequent papers.

\end{document}